\documentclass[10pt,leqno]{article}
\usepackage
[
vmargin=2.7cm,
hmargin=2.3cm
]{geometry}
\usepackage{amsmath, amsthm, amsfonts, amssymb, color}
\usepackage{enumitem}
\usepackage[numbers]{natbib}
\usepackage{tensor}
\usepackage{hyperref}
\hypersetup
{
	colorlinks,
	citecolor=blue,
	filecolor=black,
	linkcolor=blue,
	urlcolor=blue
}
\usepackage{mathrsfs}
\usepackage{color}
\theoremstyle{plain}
\newtheorem{thm}{Theorem}[section]

\newtheorem{lem}[thm]{Lemma}
\newtheorem{prp}[thm]{Proposition}

\theoremstyle{definition}
\numberwithin{equation}{section}
\newtheorem{defn}{Definition}[section]
\newtheorem{rem}{Remark}[section]
\newtheorem{Assumptions}{Assumptions}

\newtheorem{Example}{Example}
\numberwithin{Example}{section}

\newcommand{\scr}[1]{\mathscr #1}
\definecolor{wco}{rgb}{0.5,0.2,0.3}

\title{{\bf Distribution-Path Dependent Nonlinear SPDEs with Application to Stochastic Transport Type Equations}\footnote{Feng-Yu Wang is supported in
 part by  NNSFC (11831014, 11921001) and the National Key R \& D Program of China (No.2020YFA0712900). Panpan Ren and Hao Tang acknowledge support by the Alexander von Humboldt Foundation.} }
\author{{\bf   Panpan Ren$^{d)}$, Hao Tang$^{c)}$, Feng-Yu Wang$^{a,b)}$  }\\
\footnotesize{$^{a)}$ Center for Applied Mathematics, Tianjin University, Tianjin 300072, China}\\
 \footnotesize{$^{b)}$ Department of Mathematics,
Swansea University, Bay Campus, Swansea SA1 8EN, United Kingdom}\\
\footnotesize{$^{c)}$ Department of Mathematics, University of Oslo, P.O. Box 1053 Blindern, N-0316 Oslo, Norway}\\
\footnotesize{$^{d)}$ Department of Mathematics, Technical University of Kaiserslautern, P.O. Box 3049  Kaiserslautern,  Germany}\\
\footnotesize{  wangfy@tju.edu.cn, F.-Y.Wang@swansea.ac.uk, haot@math.uio.no, rppzoe@gmail.com }}

\allowdisplaybreaks
\def\R{\mathbb R}  \def\ff{\frac}   
  
\def\ee{\varepsilon}
\def\dd{\delta} \def\DD{\Delta} \def\vv{\varepsilon} 
\def\<{\langle} \def\>{\rangle}  \def\gg{\gamma}
  \def\nn{\nabla} \def\pp{\partial} \def\E{\mathbb E}
\def\d{\text{\rm{d}}} \def\bb{\beta} \def\aa{\alpha} \def\D{\scr D}
  \def\si{\sigma} 
\def\F{\scr F}
\def\Lhs{\mathcal L}

\def\e{\text{\rm{e}}}  \def\OO{\Omega}  
 \def\tt{\tilde}\def\[{\lfloor} \def\]{\rfloor}
 \def\P{\mathbb P} 
\def\C{\scr C}           
   
\def\Z{\mathbb Z}  \def\ll{\lambda}
\def\L{\scr L} 
\def\i{{\rm in}}  \def\H{\mathbb H}
\def\M{\mathbb M}  \def\LL{\Lambda}
  \def\i{{\rm i}} 
\def\to{\rightarrow}\def\U{\mathbb U}
\def\W{\mathbb W}
 
\def\BB{\mathbb B} \def\i{{\rm i}}

\def\1{{\bf 1}}
\def\ss{\sqrt}
\def\beg{\begin} \def\beq{\beg{equation}}
\begin{document}

\maketitle
\abstract{By using a  regularity approximation argument, the global existence and uniqueness are derived for a class of nonlinear SPDEs depending on both the whole history and the distribution under strong enough noise. As applications, the global existence and uniqueness are proved for distribution-path dependent stochastic transport type equations, which are  arising from stochastic fluid mechanics with forces depending on the history and the environment. In particular,  the distribution-path dependent stochastic Camassa--Holm equation with or without Coriolis effect has a unique global solution when the noise is strong enough,  whereas  for the deterministic model   wave-breaking may occur.   This  indicates that
the noise   may prevent blow-up almost surely.  }

 \



\section{Introduction}\label{Section:Intro}

To describe the evolutions of stochastic systems depending on the history and micro environment, distribution-path dependent SDEs of the following type
\begin{equation}
	\label{DD} \d X(t) =  b(t,X_t,\L_{X_t})\d t + \si(t,X_t,\L_{X_t}) \d W(t),\ \ X(0)=X_0\in\R^d,\ \  t\in [0,T]
\end{equation}
have been studied intensively investigated, see for instance \cite{BRW20,HRW19, RW19, RW20, W18}  and   references therein. However, the existing study in the literature does not cover distribution-path dependent  nonlinear SPDEs containing  a singular term which is not well-defined on the state space.  The main purpose of this paper is to solve a class of such SPDEs including transport type fluid models.

Nowadays there exists an abundant amount of literature concerning the stochastic fluid models under random perturbation which we do not attempt to survey here, and we recommend  the lecture notes \cite{D-13-notes,  F-10-notes} and the monographs \cite{BFH-18-book, KS-12-book} for readers' references. On one hand, in the real world, it is  natural that the  random perturbation may rely on both the sample path due to inertia,  and  averaged  stochastic interactions from the environment. On the other hand,  to the best of our knowledge, almost nothing is known if the randomness in the stochastic fluid models also depends on the distribution and the path of unknown variables, i.e.,  distribution-path dependent stochastic fluid models. For such problems, the fundamental question on the well-posedness (even merely the existence) of solutions remains open.    Particularly, although the (distribution-path independent) stochastic transport equations have been intensively investigated  (see for example \cite{FF-13-JFA,FNO-18, FGP-10-Inven,MO-17-AMPA,MO-17-BBMS,NO-15-NODEA}),  there is no any study on distribution-path dependent stochastic transport type equations.

To study distribution-path dependent stochastic fluid models,  we may need extend \eqref{DD} to infinite dimensional case, i.e., assuming that $X$ takes value in a separable Hilbert space
$\H$. If this is the case, a singular term, which is not well-defined on $\H$, may occur and the existing study in the literature does not cover this case. More precisely, we consider the case that \eqref{DD} contains one more singular drift term $B$ taking value in a larger separable Hilbert space $\BB$ such that $\H\hookrightarrow\BB$, i.e.,
\begin{equation}
	\label{EE} \d X(t) = \left\{B(t, X(t))+ b(t,X_t,\L_{X_t})\right\}\d t + \si(t,X_t,\L_{X_t}) \d W(t),\ \
X(0)=X_0\in\H,	\ \  t\in [0,T].
\end{equation}
Indeed, when we  consider certain stochastic fluid models in Sobolev spaces $\H=H^s$, if $B(t, X(t))$ involves $\nabla X$ or some derivatives of $X$ (see Examples \ref{Exm Abs} and \ref{Exm SRCH}), then $B(t, X(t))$  may not be expected to be in $\H=H^s$. 
Particularly, when $B(t,X(t))=-(X(t)\cdot\nn)X(t)$, \eqref{EE} reduces to the following transport type equaionts 
\begin{equation}
	\label{TE}
	\d X(t) = \left\{-(X(t)\cdot\nn)X(t)+ b(t,X_t,\L_{X_t})\right\}\d t + \si(t,X_t,\L_{X_t}) \d W(t),\ 
	X(0)=X_0\in H^s,\   t\in [0,T].
\end{equation} 
We refer to Sections \ref{subsect:framework} and \ref{subsect:DST} for the precise meaning of the notations and precise setting of \eqref{EE} and \eqref{TE}, respectively. Before going further, we would like to explain that working with the  abstract framework in \eqref{EE}  
entails some difficulties:

\begin{enumerate}[label={\textbf{(\alph*)}}]
	\item\label{difficulty-global}  Since we want to cover some stochastic fluid models in the  abstract system \eqref{EE}, we only assume that the coefficients $B$, $b$ and $\si$ are locally Lipschitz in $X$. As a result,
	we do not {\it a prior} know that the solution exists globally in time. This bring us an essential difficulty. More precisely, we notice that the distribution, as a global object on the  path space, does not exist  for explosive stochastic processes whose paths are killed at the life time. As a result,
	to  investigate distribution dependent SDEs/SPDEs, we have to either consider the non-explosive setting or modify the ``distribution" by a local notion (for example,  conditional distribution given by solution does not blow up at present time).

	\item\label{difficulty-localize} Again, because the coefficients are only locally Lipschitz in $X$, we will have to localize them (by using stopping times) when we need to fix the changing Lipschitz constants. For instance, this happens when the uniqueness is considered. Then  we will be confronted with the difficulty that distribution can not be controlled by any local condition, again. And we need to identify some appropriate topology under which the distribution can be measured locally.
	
	\item\label{difficulty-Ito} Because of the singular term $B(t,X)$, compared to classical case, the It\^{o} formula is no longer available. Indeed,  to estimate $\|X\|^2_{\H}$, to use the It\^{o} formula  in a Hilbert space  (cf. \cite{Prato-Zabczyk-2014-Cambridge,Gawarecki-Mandrekar-2010-Springer}), the $\H$ inner product $ \left(B(t,X), X\right)_{\H}$ is required to be well-defined. But it is not because we only assume that $B$ takes value in $\BB\hookleftarrow\H$. Likewise, to apply the It\^{o} formula  under a  Gelfand triplet (\cite{Krylov-Rozovskiui-1979-chapter,PR-2007-book}), the dual product $\tensor[_\BB]{\langle B(t,X),X\rangle}{_{\BB^*}}$ needs to be well-defined, where $\BB^*$ is the dual space of $\BB$ with respect to $\H$. Because $\H\hookrightarrow\BB$, we see that $\BB^*\hookrightarrow\H$. However, we do not {\it a prior} know that the solution $X$ takes value in $\BB^*$ because we only assume $X(0)\in\H$.

\end{enumerate}

The first major goal  of this paper is to establish an abstract framework for  \eqref{EE}. The second goal of this work is to apply the abstract theory for \eqref{EE} to \eqref{TE}, which gives some new results for some ideal fluid systems. 

\begin{itemize}
	\item To achieve the first goal, we introduce the precise assumptions in Section \ref{subsect:framework} (see Assumption \ref{Assum-A}).
	Then we  provide our  main results  for \eqref{EE} in Theorem \ref{T1}.  The key requirements 
	for the proof are the assumption on the existence of appropriate  Lipschitz-continuous and  monotone regularizations for the singular
	term $B$.  For the difficulty \ref{difficulty-global}, in this paper we restrict our attention to the non-explosive  case only. To this end, we assume that the noise grows fast enough (cf. \ref{A3}), and then we will show that 
	the blow-up of solutions can be prevented. For the difficulty \ref{difficulty-localize}, we introduce a ``local" Wasserstein distance (see \eqref{local Wasserstein}) and assumption \ref{A5} to measure the difference of two measures, which enables us to prove the uniqueness. By introducing a mollifier satisfying certain estimates (see assumption \ref{A4}), we can overcome the difficulty \ref{difficulty-Ito}.
	
	\item With the general framework at hand, for nonlinear stochastic transport type equations,
	we are able to construct such regular approximation schemes  by using mollifying operators and establishing a commutator estimate (see Lemma \ref{Je commutator}), from which we can  
	verify the assumptions introduced Section \ref{subsect:framework} and obtain global existence and uniqueness of solutions in Sobolev spaces. This results is stated in Theorem \ref{T-TE}. 	
	Two examples of Theorem \ref{T-TE} are given. The first one,  cf. Example \ref{Exm Abs},  is a general nonlinear stochastic transport equation, and the second one is the the distribution-path dependent stochastic Camassa--Holm equation with or without Coriolis effect, cf. Example \ref{Exm SRCH}.
\end{itemize}

\subsection{A general framework}\label{subsect:framework}
Let $\H,\U$ be two separable Hilbert spaces, and let $ \Lhs_2(\U;\H)$ be the space of Hilbert-Schmidt operators from $\U$ to $\H$ with Hilbert-Schmidt norm $\|\cdot\|_{\Lhs_2(\U;\H)}$.
Throughout the paper we  fix a time $T>0$. For a Banach  space $\M$,
let  $\scr P_{T,\M}$ be  the set of probability measures on the path space $\C_{T,\M}:=C([0,T];\M)$. We also consider the weakly continuous path space
\begin{equation*}
\C_{T,\M}^w:=\left\{\xi: [0,T]\to \M \  \text{is weak continuous}\right\}.
\end{equation*}
Both $\C_{T,\M}$ and $\C_{T,\M}^w$ are Banach spaces
under the uniform norm
$$\|\xi\|_{T,\M}:= \sup_{t\in [0,T]} \|\xi(t)\|_\M.$$ Let $\scr P_{T,\M}^w$ be the space of all probability measures on $\C_{T,\M}^w$ equipped with the weak topology. Denote
$\scr P_{T,\M}= \{\mu\in \scr P_{T,\M}^w: \mu(\C_{T,\M})=1\}.$

For any  map $\xi: [0,T]\to \M$ and $t\in [0,T]$,  the path $\pi_t(\xi)$ of $\xi$ before time $t$ is given by
$$\pi_t(\xi):=\xi_t: [0,T]\to \M,\ \ \xi_t(s):= \xi(t\land s),\ \ s\in [0,T].$$ Then the  marginal distribution before time $t$ of a probability measure $\mu\in \scr P^w_{T,\M}$  reads
$$\mu_t:= \mu\circ \pi_t^{-1}.$$
Let $\L_\xi$ stand for the distribution of a random variable $\xi$. When more than one probability measure are considered, we denote $\L_{\xi}$ by $\L_{\xi|\mathbb P}$ to emphasize the reference probability measure $\P$.

The noise  $\{W(t)\}_{t\in [0,T]}$ is  a cylindrical Brownian motion on $\mathbb U$ with respect to a complete filtration probability space $(\Omega,\{\F_t\}_{t\ge 0}, \P)$, i.e.
$$W(t)=\sum_{i\ge 1} \beta^i(t) e_i,\ \ t\in [0,T]$$ for an orthonormal basis $\{e_i\}_{i\ge 1}$ of $\U$ and a sequence of independent one-dimensional Brownian motions $\{\beta^i\}_{i\ge 1}$ on $(\Omega,\{\F_t\}_{t\ge 0}, \P)$.

Consider the following nonlinear  distribution-path  dependent SPDE on $\H$:
\begin{equation}
\d X(t) = \left\{B(t, X(t))+ b(t,X_t,\L_{X_t})\right\}\d t + \si(t,X_t,\L_{X_t}) \d W(t),\ \
  \ \  t\in [0,T],
\end{equation}
where, for some separable Hilbert space $\BB$ with  $\H\hookrightarrow\hookrightarrow\BB$ ($``\hookrightarrow\hookrightarrow"$ means the embedding is compact),
\beq\label{BBS} \beg{split}
&B:  [0,T]\times \H\times \OO\to  \BB,\\
&b:  [0,T]\times \C_{T,\H}^w\times \scr P_{T,\H}^w\times \OO\to  \H,\\
&\si:  [0,T]\times \C_{T,\H}^w\times \scr P_{T,\H}^w\times \OO  \to \Lhs_2(\U;\H)
\end{split}\end{equation}
are progressively measurable maps.

\begin{defn}  (1) A progressively measurable    process $X_T:=\{X(t)\}_{t\in [0,T]}$ on $\H$ is called a solution of \eqref{EE}, if  it is continuous in $\BB$ and   $\P$-a.s.
	$$  X(t) = X(0)+\int_0^{t}  \left\{B(s, X(s))+ b(s,X_s,\L_{X_s})\right\}\d s + \int_0^{t} \si(s,X_s,\L_{X_s}) \d W(s),\ \ t\in [0,T],$$
	where 	$\int_0^{t} \left\{B(s, X(s))+ b(s,X(s),\L_{X_s})\right\}\d s$ is the Bochner integral on $\BB$ and
	 $t\mapsto \int_0^{t} \si(s,X(s),\L_{X_s}) \d W(s)$ is a continuous local martingale on $\H$.
	
	(2) A couple $(\tt X_T, \tt W_T)=(\tt X(t), \tt W(t))_{t\in [0,T]}$ is called a weak solution of \eqref{EE}, if there exists a complete filtration probability space $(\tt\OO,\{\tt\F_t\}_{t\ge 0},\tt \P)$ such that $\tt W_T$ is a cylindrical Brownian motion on $\mathbb U$ and $\tt X_T$ is a solution of \eqref{EE} for $(\tt W_T, \tt \P)$ replacing $(W_T, \P)$.
\end{defn}

Since both $X(t)$ and    $\int_0^{t}   b(s,X_s,\L_{X_s})\d s+ \int_0^{t} \si(s,X_s,\L_{X_s}) \d W(s)$ are stochastic processes on $\H$, so is
	$\int_0^t B(s,X(s))\d s$, although $B(s,X(s))$ only takes values in $\BB$.

To ensure the non-explosion such that the distribution is well defined, we will take a Lyapunov type condition \ref{A3} below.  We write $V\in\scr V$, if $V\in C^2([0,\infty);[0,\infty))$ satisfies
$$ V(0)=0,\  \ V'(r)>0{\rm\ and\ }   V''(r)\le 0  \ {\rm for\ } r\ge 0,   \ V(\infty):=\lim_{r\to\infty} V(r)=\infty.$$
 Consider  the  following ``Wasserstein distance" induced by $V\in\scr V$:
$$\W_{2,\M}^{V}(\mu,\nu):= \inf_{\pi\in \scr C(\mu,\nu)} \int_{{\C^w_{T,\M}}\times {\C^w_{T,\M}}}    V(\|\xi-\eta\|_{T,\M}^2) \pi(\d\xi,\d\eta),\ \ \mu,\nu\in \scr P_{T,\M}^w,$$
where $\C(\mu,\nu)$ is the set of couplings of $\mu$ and $\nu$. When $V(r)=r$, $\W_{2,\M}^{V}(\cdot,\cdot)$ reduces to $\W_{2,\M}(\cdot,\cdot)^2$ which is the square of the $L^2$-Wasserstein distance   on $\scr P^w_{T,\M}$. Moreover, let
\beq\label{TXN} \tau_n^\xi=\inf\{t\ge 0:\|\xi(t)\|_\H\ge n\},\ \  n>0, \xi\in \C_{T,\H}^w. \end{equation}
Here and in the sequel, we set $\inf\emptyset =\infty$ by convention. We define the $``$local" $L^2$-Wasserstein distance by
\begin{equation}\label{local Wasserstein}
\W_{2,\BB,N}(\mu,\nu)= \inf_{\pi\in\C(\mu,\nu)} \bigg(\int_{\C_{T,\BB}\times\C_{T,\BB}}  \|\xi_{t\land \tau_N^\xi\land\tau_N^\eta}-\eta_{t\land \tau_N^\xi\land\tau_N^\eta}\|_{T,\BB}^2\pi(\d\xi,\d\eta)\bigg)^{\ff 1 2},\ \
\mu,\nu\in \scr P_{T,\BB}.
\end{equation}

We write $\mu\in \scr P_{T,\H}^V$ if $\mu\in \scr P_{T,\H}$ and
\begin{equation*}
\|\mu\|_{V}:=\int_{\C_{T,\H}}   V( \|\xi\|_{T,\H}^2) \mu(\d\xi)<\infty.
\end{equation*}
In general,  $\|\cdot\|_{V}$ may not  be a norm, but  we use this notation for simplicity.
A subset $A\subset \scr P_{T,\H}^V$ is called $V-$bounded if $\sup_{\mu\in A} \|\mu\|_V<\infty$.

   Let $T>0$ be  arbitrary.
	For any  $N>0$,    let
	\begin{align} &\C^w_{T,\H,N}=\{\xi\in\C^w_{T,\H}: \|\xi\|_{T,\H}\le N\},\ \
	\scr{P}^w_{T,\H,N}=\{\mu\in \scr P^w_{T,\H}: \mu(\C^w_{T,\H, N})=1\}. \label{Assum-notation}
	\end{align}

\begin{Assumptions}\label{Assum-A}
	Assume that $\H\hookrightarrow\hookrightarrow\BB$ is dense,  and
	there exists a dense subset $\H_0$ of $\BB^*$, the dual space of $\BB$ with respect to $\H$ such that the following conditions hold for $B,b$ and $\sigma$ in \eqref{BBS}. 	
\begin{enumerate}[label={ $(A_\arabic*)$}]
	
	\item\label{A1} $\|b(\cdot,0,\dd_0)\|_\H+\|\si(\cdot,0,\dd_0)\|_{\Lhs_2(\U;\H)}$ is   bounded on $[0,T]\times\OO$. And for any $N\ge 1$,  there exists  a constant $C_N>0$   such that   for any $\xi,\eta\in \C_{T,\H,N}$ and $\mu,\nu\in \scr P^V_{T,\H}$,	
			\begin{align*}
&\|b(t,\xi_t, \mu_t)- b(t,\eta_t,\nu_t)\|_\H+\|\si(t,\xi_t,\mu_t)- \si(t,\eta_t,\nu_t)\|_{\Lhs_2(\U;\H)}\\
&\le C_{N}\left\{\|\xi_t-\eta_t\|_{T, \H}+\W_{2,\BB}(\mu_t,\nu_t) \right\},\ \ t\in [0,T]. \end{align*}
 Next,  for any   bounded sequences $ \{(\xi^n,\mu^n)\}_{n\ge 1}\subset \C_{T,\H}\times \scr P_{T,\H}^V$ with $\|\xi^n-\xi\|_{T,\BB}\to 0$ and
	$\mu^n\to \mu$ weakly in $\scr P_{T,\BB}$ as $n\to\infty$,    we have  $\P$-a.s.
	\begin{align*} \lim_{n\to\infty}  \left\{|_{\BB} \<b(t,\xi^n, \mu_t^n)-b(t,\xi,\mu_t), \eta\>_{\BB^*} |+ \|\{\si(t,\xi^n, \mu_t^n) - \si(t,\xi,\mu_t)\}^*\eta\|_{\U}\right\}=0,\ \ \eta\in \H_0
	\end{align*}  and for any $N\ge 1$ there exists a constant $\tt C_N>0$ such that
	$$\sup_{t\in [0,T], \eta\in \C_{T,\BB,N}} \left\{\|b(t, \eta, \mu_t^n)\|_\BB+ \|\si(t,\eta,\mu_t^n)\|_{\Lhs_2(\U;\BB)}\right\}\le \tt C_N.$$

 \item\label{A2}   There exist  constants $\{C_N, C_{n,N}>0: n,N\ge 1\}$ and a sequence   of progressively measurable maps
	$$B_n: [0,T]\times\H\times \OO\to \H,\ \ n\ge 1$$  such that 	
	\begin{equation*}
	\begin{split}  & \sup_{t\in [0,T], \|x\|_\H\le N}\big(  \|B(t,x)\|_\BB+\|B_n(t,x)\|_\BB\big)\le C_N,\ \ n,N\ge 1,\\
		&\sup_{t\in [0,T], \|x\|_\H\lor\|y\|_\H\le N} \left\{\|B_n(t,x)\|_\H   +1_{\{x\ne y\}} \ff{\|B_n(t,x)-B_n(t,y)\|_\H}{\|x-y\|}\right\} \le C_{n,N},\ \ n,N\ge 1.\end{split}
	\end{equation*}
	Moreover,     for any bounded sequence  $\{\xi^n\}_{n\ge 1}$ in $\C_{T,\H}^w$
	with $\|\xi^n-\xi\|_{T,\BB} \to 0$ as $n\to\infty$, we have
	$$\lim_{n\to\infty} \int_0^T\left|
	\tensor[_\BB]{\left\<B_n(t,\xi^n(t))-B(t,\xi(t)),\eta\right\>}{_{\BB^*}}
	\right|\,\d t=0,\ \ \eta\in \H_0.$$

 \item\label{A3}  There exist $V\in \scr V$ and constants $K_1,K_2>0$ such that for any $\mu\in \scr P_{T,\H}$, $t\in [0,T],   \xi\in{\C_{T,\H}}$ and $n\ge 1,$
	\begin{align*} &V'(\|\xi(t)\|_\H^2) \left\{  2\big\< B_n(t,\xi(t))+ b(t,\xi_t,\mu_t), \xi(t)\big\>_{\H}+ \|\si(t,\xi_t,\mu_t)\|_{\Lhs_2(\U;\H)}^2 \right\}\\
	&+2 V''(\|\xi(t)\|_\H^2) \|\si(t,\xi_t,\mu_t)^*\xi(t)\|_\U^2
	\le K_1-K_2\ff{\{V'(\|\xi(t)\|_\H^2)\|\si(t,\xi_t,\mu_t)^*\xi(t)\|_\U\}^2}{ 1+V(\|\xi(t)\|_\H^2)}.\end{align*}

\item\label{A4}	There exists  a sequence of continuous linear operators $\{T_n\}_{n\ge 1}$ from $\BB$ to $\H$ with
	\begin{equation}\label{A4-1}
	\|T_nx\|_\H\le \|x\|_\H,\ \ \lim_{n\to\infty} \|T_n x-x\|_\H=0,\ \ x\in \H,
	\end{equation}
	such that  for any $N\ge 1$, there exists a constant $C_N>0$ such that
	\begin{equation}\label{A4-2}
	\sup_{\|x\|_\H\le N, n\ge 1} \left|\<T_n B(t,x), T_nx\>_\H\right|\le C_N.
	\end{equation}

\item\label{A5}
 There exist constants $K,\vv>0$ and an increasing map $C_\cdot: \mathbb N\to (0,\infty)$ such that for any $N\ge 1$, $\xi,\eta\in \C^w_{T,\H,N}$ and $\mu,\nu\in \scr P^w_{T,\H},$
	\begin{align*} &\left\<B(t,\xi(t))-B(t,\eta(t)), \xi(t)-\eta(t)\right\>_\BB\le C_N \|\xi(t)-\eta(t)\|^2_\BB,\\
&\|b(t,\xi_t, \mu_t)- b(t,\eta_t,\nu_t)\|_\BB+\|\si(t,\xi_t,\mu_t)- \si(t,\eta_t,\nu_t)\|_{\Lhs_2(\U;\BB)}\\
&\le C_{N}\left\{\|\xi_t-\eta_t\|_{T, \BB}+
\W_{2,\BB,N}\left(\mu_t,\nu_t\right)+K\e^{-\vv C_N}\left( 1\land \W_{2,\BB}(\mu_t,\nu_t)\right)\right\},\ \ t\in [0,T].
\end{align*}
	
\end{enumerate}

 \end{Assumptions}

\begin{thm}\label{T1}  Let $X_0\in L^2(\OO\to\H, \F_0,\P)$.
 \begin{enumerate} [label={\bf (\roman*)}]

 	\item\label{Thm-Existence-weak}  Assume \ref{A1}--\ref{A3}.
 	Then \eqref{EE} has a  weak  solution $(\tt X_T,\tt W_T)$ such that $\L_{\tt X(0)|\tt\P}=\L_{X_0|\P}$ and
	 \begin{equation}
	\label{Q2}\begin{split}
	\tt\E\left[ V(\|\tt X_T\|_{T,\H}^2)\right]
	\le   2K_1T +1+ \ff{64}{K_2}\left( K_1T+\tt \E[V(\|\tt X(0)\|_\H^2)]\right)<\infty.
	\end{split}
	\end{equation}
	\item \label{Thm-Coninuous}  If \ref{A4} holds, then the weak solution is continuous in $\H$.
	\item \label{Thm-Unique} If \ref{A5}  holds, then $\eqref{EE}$ has a  unique   solution  with initial value $X_0$. \end{enumerate}
	\end{thm}

	Now we give some remarks regarding the proof of Theorem \ref{T1} and Assumption \ref{Assum-A}.
\begin{rem}
		Except for the difficulties \ref{difficulty-global}, \ref{difficulty-localize} and \ref{difficulty-Ito}, we will be confronted with one additional technical obstacle. Indeed, 
	we notice that the singular term $B$ is in general not monotone in the sense of \cite{Pard} (see also \cite{PR-2007-book}). Therefore, even coming back to the distribution-path independent case, the  Galerkin  approximation under a Gelfand triple developed for quasi-linear  SPDEs does not work for the present model. 
	To overcome this obstacle, we will take a different regularization argument. The proof of Theorem \ref{T1} includes two main steps:

	\begin{description}
		\item[Step 1: regular case] We first establish the solvability of the regular case, i.e., $B=0$ (see Proposition \ref{PR1}). In this step, we need \ref{A1}  as \ref{A1} describes the local Lipschitz continuity of the regular coefficients $b(t,\xi,\mu)$ and $\si(t,\xi,\mu)$ in $(\xi,\mu)$ under the metric induced by $\|\cdot\|_\H$ and $\W_{2,\BB}$.  Recalling the difficulty \ref{difficulty-global} mentioned before, we 
		restrict our attention to the non-explosive case.  Hence we need the assumption \ref{A3} which  is a Lyapunov type condition
		 ensuring the global existence of the solution.  Furthermore,
		\ref{A5} means that the dependence on the distribution of the coefficients
		 is asymptotically determined by the distribution of local paths, and it will be used to prove the pathwise uniqueness. Actually, \ref{A5} is proposed to overcome the difficulty \ref{difficulty-localize}.  
		
		\item[Step 2: singular case]  
		Then we will propose a regularization argument to establish existence and uniqueness to \eqref{EE}. Therefore in \ref{A2} we assume that the singular term $B\in\BB$ can be approximated by a regular term $B_n\in\H$ with certain nice properties. 		
		The result in \textbf{Step 1} guarantees that the approximation problem (see \eqref{P11}, where $B$ in \eqref{EE} is replaced by $B_n$) can be uniquely solved on $[0,T]$ for any given $T>0$ and we refer to Proposition \ref{PR1}.  				
		 Then we use the martingale approach to pass limit to the original problem \eqref{EE}, where we need the continuity of the coefficient in $\mu$ under the weak topology (see \ref{A1}). Precisely speaking,  by Prokhorov's theorem and Skorokhod's theorem, we can get almost sure convergence of the approximation solutions  relative to a new probability space. Then by the martingale representation theorem, we can identify the limit of the stochastic integral. Finally we establish the uniqueness, which together with the Yamada-Watanabe type result gives the existence and uniqueness of a pathwise solution. Finally, as mentioned before, the It\^{o} formula can not be applied to to $\|X(t)\|^2_{\H}$ directly (see difficulty \ref{difficulty-Ito}). Hence it is not obvious to obtain the time continuity of the solution in $\H$. And we need to mollify the equation first by using some mollifiers. Hence \ref{A4} provides certain properties of such mollifiers.
		
	\end{description}

\end{rem}

\subsection{Distribution-path dependent stochastic transport type equations}\label{subsect:DST}
Let $d\geq1$ and $\mathbb{T}^d=(\R/2\pi\Z)^d$ be the $d$-dimensional torus.
Let $\DD$ be the Laplacian operator on $\mathbb T^d$, and let ${\rm i}$ denote  the imaginary unit.
Then $\{\e^{{\rm i} \<k,\cdot\>}\}_{k\in \mathbb Z^d}$ consists of an eigenbasis of the Laplacian $\DD$ in the complex $L^2$-space of the normalized   volume measure $\mu(\d x):= (2\pi)^{-d}\d x$ on $\mathbb T^d$:
$$\DD \e^{{\rm i}\<k,\cdot\>} = -|k|^2 \e^{{\rm i}\<k,\cdot\>},\ \ k\in \mathbb Z^d.$$
For a function $f\in L^2(\mu)$, its Fourier transform is given by
$$\widehat{f}(y):=\F(f) (y)= \mu(f\e^{{\rm i}\<y,\cdot\>})= \int_{\mathbb T^d} f \e^{{\rm i}\<y,\cdot\>}\d \mu,\ \ y\in \R^d.$$
It is well known that
\beg{equation}\label{X0} \|f\|_{L^2(\mu)}^2= \sum_{k\in\mathbb Z^d} |\widehat{f}(k)|^2,\ \ f\in L^2(\mu),\end{equation} and
\beg{equation}\label{XXX}
\sum_{m\in\Z^d} \widehat{g}(k-m) \widehat{f}(m)= \widehat{fg}(k),\ \ k\in\Z^d, f,g\in L^4(\mu).\end{equation}
By the spectral representation, for any $s\ge 0$, we have
\beg{align*} &D^sf:= (I-\DD)^{\ff s 2} f=  \sum_{k\in \mathbb Z^d} (1+|k|^2)^{\ff s 2} \widehat {f} (k) \e^{{\rm i}\<k,\cdot\>},\ \ k\in \mathbb Z^d,\\
& f\in \D(D^s):=  \left\{f\in L^2(\mu): \|D^sf\|_{L^2(\mu)}^2= \sum_{k\in\mathbb Z^d} (1+|k|^2)^s |\widehat f(k)|^2<\infty\right\}.\end{align*}
Then
$$H^s:= \{f=(f_1,\cdots, f_d): f_i\in \D(D^s), 1\le i\le d\}$$   is a separable Hilbert space with inner product
$$\<f,g\>_{H^s}:= \sum_{i=1}^d  \<D^s f_i,D^s g_i\>_{L^2(\mu) } = \sum_{k\in \mathbb Z^d} (1+|k|^2)^s \<\widehat f(k), \widehat g(k)\>_{\R^d}.$$

 Now,  we recall the stochastic transport SPDE \eqref{TE}  on $H^s$:
\begin{equation}
\d X(t) = \left\{-(X(t)\cdot\nn)X(t)+ b(t,X_t,\L_{X_t})\right\}\d t + \si(t,X_t,\L_{X_t}) \d W(t),\ \
 \ \  t\in [0,T],
\end{equation} where $W(t)$ is the cylindrical Brownian  motion on $\mathbb U:= L^2(\mathbb T^d\to \R^d)$, and
$$ b: [0,T]\times \C_{T,H^s}^w\times \scr P_{T,H^s}^w\times \OO\to H^s,\ \ \si:  [0,T]\times \C_{T,H^s}^w\times \scr P_{T,H^s}^w\times \OO\to \Lhs_2(\U;H^s)$$
are measurable.

To apply Theorem \ref{T1}, we make the following assumptions on $b$ and $\si$.

\begin{Assumptions}\label{Assum-B}
 Let  $d\ge1$,  $V\in \scr V$, $s>\ff d 2 +2$, $s'=s-1$ and $T>0$ be  arbitrary.
We assume that the following conditions hold for $\H=H^s$ and $\BB=H^{s'}$.

\begin{enumerate}[label={ $(B_\arabic*)$}]
	\item\label{B1}  Conditions in  \ref{A1} hold. 	
	
	\item\label{B2}
	There exist  constants  $K_1,K_2>0$ such that for any $\mu\in \scr P_{T,\H}$, $t\in [0,T],   \xi\in{\C_{T,\H}}$ and $n\ge 1,$
	\begin{align*}
	&V'(\|\xi(t)\|_\H^2) \left\{ 2 K_0 \|\xi(t)\|_{\BB}\|\xi(t)\|^2_{\H}+ 2\big\< b(t,\xi_t,\mu_t), \xi(t)\big\>_{\H}+ \|\si(t,\xi_t,\mu_t)\|_{\Lhs_2(\U;\H)}^2 \right\}\\
	&+2 V''(\|\xi(t)\|_{\H}^2) \|\si(t,\xi_t,\mu_t)^*\xi(t)\|_\U^2
	\le K_1-K_2\ff{\{V'(\|\xi(t)\|_{\H}^2)\|\si(t,\xi_t,\mu_t)^*\xi(t)\|_\U\}^2}{ 1+V(\|\xi(t)\|_{\H}^2)}.
	\end{align*}

	\item\label{B3}   There exist constants $K,\vv>0$ and an increasing map $C_\cdot: \mathbb N\to (0,\infty)$ such that for any $N\ge 1$, $\xi,\eta\in \C^w_{T,\H,N}$ and $\mu,\nu\in \scr P^w_{T,\H},$
	\begin{align*}
&\|b(t,\xi_t, \mu_t)- b(t,\eta_t,\nu_t)\|_{ \BB}+\|\si(t,\xi_t,\mu_t)- \si(t,\eta_t,\nu_t)\|_{\Lhs_2(\U; \BB)}\\
&\le C_{N}\left\{\|\xi_t-\eta_t\|_{T,\BB}+
\W_{2,\BB,N}\left(\mu_t,\nu_t\right)+K\e^{-\vv C_N}\left( 1\land \W_{2,\BB}(\mu_t,\nu_t)\right)\right\},\ \ t\in [0,T].
\end{align*}
\end{enumerate}
\end{Assumptions}

Then we have  the following result:

\begin{thm}\label{T-TE} Assume $s>\frac{d}{2}+2$, \ref{B1} and \ref{B2}. For any $X_0\in L^2(\OO\to H^s, \F_0,\P)$,  \eqref{TE}  has a weak  solution $(\tt X_T,\tt W_T)$ such that $\L_{\tt X(0)|\tt\P}= \L_{X_0|\P}$, $\tt X_T$ is continuous in $H^s$ and
	\begin{equation}
	\label{Q-TE}\begin{split}
	\tt \E\left[ V(\|\tt X_T\|_{T,H^s}^2)\right]
	\le  2K_1T +1+ \ff{64}{K_2}\left( K_1T+\tt \E[V(\|\tt X(0)\|_{H^s})]\right).
	\end{split}
	\end{equation}
If, moreover, \ref{B3} holds, then \eqref{TE} has a unique solution.
\end{thm}

Below we give some remarks concerning Theorem \ref{T-TE}.

\begin{rem}
	 We first notice that \eqref{TE} does not contain the viscous term $\DD X(t)$, which provides additional regularization effect to make the problem of existence easier, see \cite[Chapter 5]{DP}. 
	 Besides the existence and uniqueness, it is interesting to clarify the effect of noise on the properties of solutions. We notice that existing results on the regularization effects by noises for transport type equations are mainly for linear equations or for linear growing noises, see for instance \cite{FF-13-JFA, FNO-18,  FGP-10-Inven,K-10-JFA,MO-17-BBMS,NO-15-NODEA} for linear transport equations, and \cite{FGP-11-SPA,GV-14-AoP,RZZ-14-SPA,T-18-SIMA} for linear noise. For nonlinear equations with nonlinear noise, there are  examples with positive answers showing that noises can be used to regularize singularities caused by nonlinearity.
	 For example, for the stochastic 2D Euler equations, coalescence of vortices may disappear   \cite{FGP-11-SPA}.  But there are also counterexamples such as the fact that noise does not prevent shock formation  in the Burgers equation, see \cite{F-10-notes}.
	 Therefore, for nonlinear  SPDEs, what kind of nonlinear noise can prevent blow-up is a question worthwhile to study. In the current work, the main idea is to use the
	 stochastic part of the equation to avoid any blow-up phenomena that could arise under the presence of the singular drift. Hence we use the Lyapunov type condition \ref{B2} to measure how strong the noise term needs to be (see also \cite[Theorem
	 III.4.1]{Hasminskii-1969-Book} for the finite dimensional case and  \cite{Brzezniak-etal-2005-PTRF} for the stochastic nonlinear beam equations). In this way, the noise effect given by the large enough noise is macroscopic and it is different from many previous works, where small noise can also bring regularization effect, see for example \cite{FGP-10-Inven,FGP-11-SPA}. Besides, the noise structure in \cite{FGP-10-Inven,FGP-11-SPA} are transport noise in the Stratonovich sense.  \textit{A priori}, it is not clear how to interpret the equation \eqref{TE}. In the current work, our main interest are mainly mathematical and we believe that 
	searching for nonlinear noise such that blow-up can be prevented is important because it helps us understand the regularizing mechanisms of noise. This in turn brings us one further
	 step  to find the really correct and physical noise which provides such regularization.

\end{rem}

\begin{rem}
	We remark here that there is a gap between the index $s>\frac{d}{2}+2$ in Theorem \ref{T-TE} and the critical value $s>\frac{d}{2}+1$ such that $H^s\hookrightarrow W^{1,\infty}$. Formally speaking, on one hand, because the transport term $(u\cdot\nabla) u$ loses one order of regularity,  we have to consider uniqueness in $H^{s'}$ with $s'\leq s-1$, i.e., we ask $\BB=H^{s'}$ in \ref{B3}. One the other hand, since $\<(u\cdot\nabla) u,u\>_{H^s}\leq c_s\|u\|_{W^{1,\infty}}\|u\|^2_{H^s}$ for smooth $u$, to verify \ref{B2}, we have to pick $s'\leq s-1$ such that $\BB=H^{s'}\hookrightarrow W^{1,\infty}$. Therefore we have to require $s-1>\frac{d}{2}+1$. However, if we  only consider  local solutions in $H^s$ without assuming  \ref{B2}    (as is explained before, in this case  the distribution has to be modified), then $s>\frac{d}{2}+1$ will be enough.
\end{rem}

To conclude this section, we present below two examples  to illustrate Theorem \ref{T-TE}.

\begin{Example}\label{Exm Abs}
Let $s,s'=s-1$ be in assumption \ref{Assum-B} and take $\mathbb U= H^s$.  Let  $\mu(F)= \int F\d\mu$ for $F\in L^1(\mu)$, and take
$$ b(t,\xi,\mu)= h(t,\|\xi\|_{H^{s'}}, \mu(F_b))\xi(t),\ \ \si(t,\xi,\mu)= \bb  (1+\|\xi\|_{T,H^{s'}})^\aa\<\xi(t),\cdot\>_{H^{s}}  x_0+\si_0(t, \|\xi\|_{H^{s'}},\mu(F_\si)),$$
where $\aa, \bb>0$ are constants to be determined, and \beg{enumerate}
\item[(1)]   $ x_0\in H^s$ with $\|x_0\|_{H^{s}}=1$ is a fixed element;
\item[(2)] $F_{b}, F_\si: \C_{T,H^{s'}}\to \R^m$ are bounded and Lipschtiz continuous for some $m\ge 1$;
\item[(3)] $h(t,\cdot,\cdot):  \R\times \R^{m}\to \R$ is locally Lipschtiz continuous uniformly in $t\in [0,T]$ such that
$$\sup_{(t,z)\in [0,T]\times \R^m, |x|\le r} |h(t,x,z)| \le  c (1+r^{2\aa} ),\ \ r\ge 0$$ holds for some constant $c>0;$
\item[(4)]  $\si_0(t,\cdot,\cdot):    \R\times \R^m\to \Lhs_2(H^s; H^s)$  is bounded and  locally Lipschtiz continuous uniformly in $t\in [0,T]$.
\end{enumerate}
 If $\aa\ge \ff 1 2$ and $\bb$ is large enough, then  for any probability measure $\mu_0$ on $H^s$ with $\mu_0(\|\cdot\|_{H^s}^2)<\infty$,  $\eqref{EE}$ has a weak  solution $(\tt X_T,\tt W_T)$ with $\L_{\tt X(0)|\tt\P}= \mu_0$,  which is continuous in $H^s$ and satisfies
 $$\tt\E\left[\ \log(1+\|\tt X_T\|_{T,H^s}^2)\right]<\infty.$$
 In particular, if $m=1$ and $F_b(\xi)= F_\si(\xi)= \|\xi\|_{T,H^{s'}}\land R$ for some constant $R>0$, then for any $X(0)\in L^2(\OO\to H^s,\F_0,\P)$,   $\eqref{EE}$  has a unique solution, which is continuous in $H^s$ and satisfies  $$\E\left[\log(1+\|X_T\|_{T,H^s}^2)\right]<\infty.$$
\end{Example}

\begin{proof}[Proof of Example \ref{Exm Abs}]
Let $\aa \ge \ff 1 2$, and take $V(r)= \log(1+r)\in\scr V$. By Theorem \ref{T-TE}, we only need to verify conditions \ref{A1}, \ref{B2} with $\H=\mathbb U=H^s$, $\BB=H^{s'}$, $\H_0=H^{s+1}$  and large enough $\bb>0$,
and finally prove \ref{B3} with $m=1$ and $F_b(\xi)=F_\si(\xi)= \|\xi\|_{T,H^s}\land R$.

To begin with,  it is easy to see that the weak convergence in $\scr P_{T, \BB}$ is equivalent to that in the metric
$$\W_{1,\BB}(\mu,\nu):= \inf_{\pi\in \C(\mu,\nu)} \int_{\C_{T,\BB}\times \C_{T,\BB}} (1\land \|\xi-\eta\|_{T,\BB}) \pi(\d\xi,\d\eta).$$
Then  (1)-(4) and $\H\hookrightarrow\BB$ imply that for any $N\ge 1$ there exists a constant $C_N>0$ such that  for all $\eta\in H^{s+1}$,
\beg{align*} &\|b(t,\xi,\mu)- b(t,\eta,\nu)\|_{\H}
+ \|\sigma(t,\xi,\mu)- \sigma(t,\eta,\nu)\|_{\Lhs_2(\H; \H)}\le C_N\left( \|\xi-\eta\|_{T,\H}+ \W_{1,\BB}(\mu,\nu)\right).
\end{align*}
Therefore, \ref{A1} holds.

Next,  let $C= \sup_{(t,r,z)\in [0,T]\times [0,\infty)\times\R^m}  \|\si_0(t,r,z)\|_{\Lhs_2(\H;\H)}^2$. We have
\begin{align*}
&V'(\|\xi(t)\|_\H^2) \left\{ 2 K_0 \|\xi(t)\|_{\BB}\|\xi(t)\|^2_{\H}+ 2\big\< b(t,\xi_t,\mu_t), \xi(t)\big\>_{\H}+ \|\si(t,\xi_t,\mu_t)\|_{\Lhs_2(\U;\H)}^2 \right\}\\
&\le \ff{2K_0 \|\xi(t)\|_\BB\|\xi(t)\|_{\H}^2+ \ff {5 \bb^2} 4 (1+ \|\xi_t\|_{T,\BB}^\aa)^2\|\xi(t)\|_\H^2 +5C}{1+\|\xi(t)\|_\H^2}\\
&\le \ff {\|\xi(t)\|_\H^2}{1+\|\xi(t)\|_\H^2} \left\{C_1(1+\|\xi_t\|_{T,\BB}^{2\aa}) +\ff {5\bb^2} 4 (1+ \|\xi_t\|_{T,\BB}^\aa)^2\right\}
\end{align*}
for some constant $C_1>0$, and on the other hand,
\beg{align*} &2 V''(\|\xi(t)\|_{\H}^2) \|\si(t,\xi_t,\mu_t)^*\xi(t)\|_\U^2\le   - \ff{2 \|\xi(t)\|_\H^4}{ (1+\|\xi(t)\|_\H^2)^2} \left\{\ff {3 \bb^2}4 (1+\|\xi_t\|_{T,\BB}^\aa)^2  - 4 C\right\}\\
&\ff{\{V'(\|\xi(t)\|_{\H}^2)\|\si(t,\xi_t,\mu_t)^*\xi(t)\|_\U\}^2}{ 1+V(\|\xi(t)\|_{\H}^2)}  \le \ff{\|\xi(t)\|_\H^4}{ (1+\|\xi(t)\|_\H^2)^2} \left\{ \bb^2 (1+\|\xi_t\|_{T,\BB}^\aa)^2  +2 C\right\}.\end{align*}
Therefore, when $\bb>2\ss{C_1}$, \ref{B2} holds for some constants $K_1,K_2>0$.

Finally, let   $m=1, F_b(\xi)=F_\si(\xi)= \|\xi\|_{T, \BB}\land R$.  It suffices to verify \ref{B3} for $N\ge R$. In this case, by the formulation of $b,\si$ and conditions (1)-(4),   for any $N\ge R$,  there exists a constant $C_N>0$ such that
\begin{equation}\label{ERR}
\beg{split}
&\|b(t,\xi,\mu)- b(t,\eta,\nu)\|_{\BB}+ \|\si(t,\xi,\mu)- \si(t,\eta,\nu) \|_{\Lhs_2(\H;\BB)}\\
&\le C_N\left( \|\xi-\eta\|_{T,\BB}+ |\mu_t(\|\cdot\|_{T,\BB}\land R)- \nu_t(\|\cdot\|_{T,\BB}\land R) |\right).
\end{split}
\end{equation}
Denote
$$\|\xi-\eta\|_{\tau_N}=\sup_{t\in [0,T\land\tau_N^\xi\land\tau_N^\eta]}\|\xi(t)-\eta(t)\|_\BB.$$
When $N\ge R$ we have
$$\left|\|\xi_t\|_{T,\BB}\land R- \|\xi_t\|_{T,\BB}\land R\right| \beg{cases} \le  \|\xi_t-\eta_t\|_{T,\BB}= \|\xi-\eta\|_{\tau_N}, &\text{if}\ \tau_N^\xi\land\tau_N^\eta>t,\\
= R- \|\eta_t\|_{T,\BB}\land R \le  \|\xi-\eta\|_{\tau_N}, &\text{if}\ \tau_N^\xi\le t,  \tau_N^\eta>t\\
= R- \|\xi_t\|_{T,\BB}\land R \le  \|\xi-\eta\|_{\tau_N}, &\text{if}\ \tau_N^\xi > t,  \tau_N^\eta\le t\\
= 0 \le  \|\xi-\eta\|_{\tau_N}, &\text{if}\ \tau_N^\xi\lor \tau_N^\eta \le t. \end{cases} $$
Consequently,
$$ |\mu_t(\|\cdot\|_{T,\BB}\land R)- \nu_t(\|\cdot\|_{T,\BB}\land R)|\le \inf_{\pi\in\C(\mu_t,\nu_t)} \int_{\C_{T,\BB}\times\C_{T,\BB}} \|\xi-\eta\|_{\tau_N}\d\pi
\le \W_{2,\BB,N}(\mu_t,\nu_t),$$
so that \eqref{ERR} implies \ref{B3} for $K=0$.
\end{proof}

\begin{Example}\label{Exm SRCH}  Now we consider a family of stochastic models which are more physical relevant.  Let $s,s'$ be in assumption \ref{Assum-B} with $d=1$ and take $\mathbb U= H^s$.  We focus on the following PDE
		\begin{equation}\label{RCH}
\partial_tu+ u\partial_xu+(1-\partial^2_{xx})^{-1}\partial_x\left(a_0u+a_1 u^2+a_2(\partial_xu)^2+a_3 u^3 +a_4 u^4\right)=0,
	\end{equation}
	where $a_i$  ($i=0,1,2,3,4$) are some constants.
	Before we consider their stochastic versions, we briefly recall some background of \eqref{RCH}.
	Due to the abundance of literature on \eqref{RCH}, here we only mention a few related results.
	 If $a_1=1$, $a_2=\frac12$ and $a_0=a_3=a_4=0$,
	\eqref{RCH} becomes the Camassa--Holm equation
	\begin{equation}
\partial_tu+u\partial_xu+(1-\partial_{xx}^2)^{-1}\partial_x\left(u^2+\frac12 (\partial_xu)^2\right)=0.\label{CH}
	\end{equation}
	Equation \eqref{CH}   models the unidirectional propagation of shallow water waves over a flat bottom and it appeared initially in the context of hereditary symmetries studied by Fuchssteiner and Fokas \cite{FF-1981-PhyD} as a bi-Hamiltonian generalization of KdV equation. Later, Camassa and Holm \cite{CH-1993-PRL} derived it by approximating directly in the Hamiltonian for Euler equations in the shallow water regime.
	It is well known that \eqref{CH}
	exhibits both phenomena of (peaked) soliton interaction and wave-breaking.
	When $a_1=\frac{b}{2}$, $a_2=\frac{3-b}{2}$ with $b\in\R$ and $a_0=a_3=a_4=0$,
	\eqref{RCH} reduces to the so-called  $b$-family equations, cf. \cite{HS-2004-PhyA,CL-2009-ARMA},
	\begin{equation}
	\partial_tu+u\partial_xu+(1-\partial_{xx}^2)^{-1}\partial_x\left(\frac{b}{2}u^{2}+\frac{3-b}{2}(\partial_xu)^2\right)=0.\label{b family}
	\end{equation}
When $a_0\in\R$,  $a_1=1$ , $a_2=\frac12$ and $a_3=a_4=0$, \eqref{RCH} is a dispersive  evolution equation derived by Dullin et al.~in \cite{Dullin-Gottwald-Holm-2001-PRL}   as a model  governing planar  solutions  to Euler's equations in the shallow--water regime.  Finally, when $a_i$ ($i=0,1,2,3,4$) are suitably chosen, \eqref{RCH} becomes the recently derived rotation-Camassa--Holm equation describing the motion of the fluid with the Coriolis effect from the incompressible shallow water in the equatorial region, cf. \cite[equation (4.9)]{GLS-2019-JMFM}. In this case, $a_3\neq0$ and $a_4\neq0$ so that the equation has a cubic and quartic nonlinearities.
	
	For this family of PDEs, if distribution-path dependent noise is involved, which can be explained as the weakly random dissipation, cf. \cite{T-18-SIMA}, we consider
	\begin{equation}\label{SRCH}
\d u(t)+\left[u\partial_xu+(1-\partial^2_{xx})^{-1}\partial_x\left(a_0u+a_1 u^2+a_2 (\partial_xu)^2+a_3 u^3 +a_4 u^4\right)\right](t)\d t=\si(t,u_t,\L_{u_t}) \d W(t),
\end{equation}
where
$$\si(t,u,\mu)=\bb  (1+\|u\|_{T,H^{s'}})^\aa\<u(t),\cdot\>_{H^{s}}\cdot v+\si_0(t, \|u\|_{H^{s'}},\mu(F_\si)),$$
and $v\in H^s$ is a fixed element such that $\|v\|_{H^{s}}=1$  and $\si_0$ satisfies  condition (4) with $m=1$ as in Example \ref{Exm Abs}. 
Let
$$F(u)=(1-\partial^2_{xx})^{-1}\partial_x\left(a_0u+a_1 u^2+a_2 (\partial_xu)^2+a_3 u^3 +a_4 u^4\right).$$
It is easy to show that there is a constant $C>0$ such that 
$$	\|F(u)\|_{H^s}\leq C\left(|a_0|+(|a_1|+|a_2|)\|u\|_{W^{1,\infty}} +|a_3|\|u\|^2_{W^{1,\infty}}+ |a_4|\|u\|^3_{W^{1,\infty}}\right)\|u\|_{H^s},$$
and
$$
\|F(u)-F(v)\|_{H^{s'}}
\leq C
\left[|a_0|+(|a_1|+|a_2|)I_{s}(u,v)+|a_3| I_{s}^2(u,v)+|a_4| I_{s}^3(u,v)\right]
\|u-v\|_{H^s}$$
with $I_s(u,v)=\|u\|_{H^{s}}+\|v\|_{H^{s}}$. Since $H^{s'}\hookrightarrow W^{1,\infty}$, $F(\cdot)$ satisfies the the  estimates for drift part as  in \ref{B1} and \ref{B3}.
Going along the lines as
in the proof of Example \ref{Exm Abs} with minor modification, we can see that if $\beta>1$ is large enough and
	\begin{equation*}
\alpha
\begin{cases}
\ge3/2, & {\rm if}  \ a_4\neq0,\ a_0,a_1,a_2,a_3\in\R\ \text{(with Coriolis effect)},\\
\ge1, &  {\rm if} \  a_4=0,\ a_3\neq 0,\ a_0,a_1,a_2\in\R,\\
\ge1/2, &  {\rm if} \ a_3=a_4=0,\ a_1\neq0,\ a_2\neq0,\ a_0\in\R\ \text{(without Coriolis effect)},\\
\end{cases}
\end{equation*}
then for any $u(0)\in L^2(\OO\to H^s,\F_0,\P)$,   \eqref{SRCH} has a unique solution with continuous path in $H^s$ and  $$\E\left[\log(1+\|u_T\|_{T,H^s}^2)\right]<\infty.$$
Therefore, in contrast to the  deterministic case where wave-breaking phenomenon may occur in finite time, see \cite{CE-1998-Acta,CE-1998-CPAM,ZLM-20-AMPA},
the blow-up  is prevented   when the growth of the noise coefficient  in \eqref{SRCH}  is faster enough.
\end{Example}

\

The remainder of the paper is organized as follows. In Section 2, we consider the regular case where $B=0$. Then we prove Theorem \ref{T1} and Theorem \ref{T-TE} in Section 3
and Section 4 respectively.

\section{Regular case: $B=0$}\label{Section:B=0}

We consider the following distribution-path dependent SPDE:
\begin{equation}
\label{E0}
\d X(t) =  b(t,X_t,\L_{X_t})\d t + \si(t,X_t,\L_{X_t}) \d W(t),\ \ X(0)=X_0,\ \ t\in [0,T].
\end{equation}
Recall that
$$\W_{2,\H} (\mu,\nu):= \inf_{\pi\in \C(\mu,\nu)} \bigg(\int_{\C_{T,\H}^w\times\C_{T,\H}^w} \|\xi-\eta\|_{T,\H}^2\,\pi(\d\xi,\d\eta)\bigg)^{\ff 1 2},\ \ \mu,\nu\in \scr P_{T,\H}^w.$$
Then assumption \ref{Assum-A}  for $B=0$  implies the following assumption \ref{Assum-C}:

\begin{Assumptions}\label{Assum-C}
	 With the same notation as in \eqref{Assum-notation}, we assume the following,  for some Hilbert space $\BB$ with dense and compact embedding $\H\hookrightarrow\hookrightarrow \BB$:
\begin{enumerate}[label={ $(C_\arabic*)$}]
	\item\label{C1}     For any   $N\ge 1$, there exists  a constant $C_N>0$   such that   for any $\xi,\eta\in \C_{T,\H,N}$ and $\mu,\nu\in \scr P^V_{T,\H}$, we have that $\P$-a.s. for $t\in [0,T]$,
	\begin{align*}
	&\|b(t,\xi_t, \mu_t)\|_\H+\|\si(t,\xi_t,\mu_t)\|_{\Lhs_2(\U;\H)}\le C_N,
	\end{align*}
	\begin{align*}
	&\|b(t,\xi_t, \mu_t)- b(t,\eta_t,\nu_t)\|_\H +\|\si(t,\xi_t,\mu_t)- \si(t,\eta_t,\nu_t)\|_{\Lhs_2(\U;\H)}
	\le C_{N}\left\{\|\xi_t-\eta_t\|_{T,\H}+\W_{2,\BB}(\mu_t,\nu_t) \right\}.\end{align*}

	\item\label{C2} There exists a dense subset $\H_0\subset\H$ such that
for any bounded sequence $\{(\xi^n,\mu^n)\}_{n\ge 1}\subset \C_{T,\H}\times\scr P_{T,\H}^V$ with $\|\xi^n-\xi\|_{T,\H}\to 0$ and
	$\mu^n\to \mu$ weakly in $\scr P_{T,\BB}$ as $n\to\infty$,   we have
	$$\lim_{n\to\infty}   \big\{|\<b(t,\xi^n, \mu_t^n)- b(t,\xi,\mu_t),\eta\>_\H +\|\{\si(t,\xi^n, \mu_t^n)- \si(t,\xi,\mu_t)\}^*\eta\|_{\U}\big\}=0,\ \ \eta\in \H_0.$$
	
	\item\label{C3}    There exist  constants $K_1,K_2>0$ such that for any $\mu\in \scr P_{T,\H}$, $t\in [0,T] $ and $\xi\in{\C_{T,H}}$,
	\begin{align*}
	&V'(\|\xi(t)\|_\H^2) \left\{  2\< b(t,\xi_t,\mu_t), \xi(t)\>_{\H}+ \|\si(t,\xi_t,\mu)\|_{\Lhs_2(\U;\H)}^2\right\}\\
	&+2 V''(\|\xi(t)\|_\H^2) \|\si(t,\xi_t,\mu_t)^*\xi(t)\|_\U^2
	\le K_1-K_2\ff{\{V'(\|\xi(t)\|_\H^2)\|\si(t,\xi_t,\mu_t)^*\xi(t)\|_\U\}^2}{ 1+V(\|\xi(t)\|_\H^2)}.
	\end{align*}
	
	\item\label{C4}   There exist constants $K,\vv>0$, an increasing  map $C_\cdot:\mathbb N\to (0,\infty)$,such that
  for any $\xi, \eta\in{\C_{T,\H, N}} $
	\begin{align*}
	&\|b(t,\xi_t, \mu_t)- b(t,\eta_t,\nu_t)\|_\BB+\|\si(t,\xi_t,\mu_t- \si(t,\eta_t,\nu_t)\|_{\Lhs_2(\U;\BB)}\\
	&\le C_{N}\left\{\|\xi_t-\eta_t\|_{\BB}+
	\W_{2,\BB,N}\left(\mu_t,\nu_t\right)+K\e^{-\vv C_N}\left( 1\land \W_{2,\BB}(\mu_t,\nu_t)\right)\right\},\ \ t\in [0,T].\end{align*}
	\end{enumerate}

\end{Assumptions}

The main result of this section is the following.

\begin{prp}\label{PR1}    Assume \ref{C1}--\ref{C3}.  For any $T>0$ and $X_0\in L^2(\OO\to\H,\F_0,\P)$,
	\eqref{E0} has a  solution $X\in C([0,T];\H)$ and satisfies
	\begin{equation}\label{E0 estimate}
	\E\left[ V(\|X_T\|_{T,\H}^2) \right]
	\le  2K_1T +1+ \ff{64}{K_2}\left( K_1T+\E\left[V(\|X_0\|_\H^2)\right]\right)<\infty. \end{equation}
	Moreover, if  \ref{C4} holds, then the solution   is unique.
\end{prp}

To prove this result, we first consider the global monotone situation, and then extend to the local case.

\begin{lem}\label{L1} Let $b(t,\xi,\mu)$ and $\si(t,\xi,\mu)$ be  continuous in $(\xi,\mu)\in \C_{T,\H}\times\scr P_{T,\H}$.  If there exists a positive random variable $\gg$ with $\E[\gg]<\infty$ and a constant $K>0$, such that for any ${\C_{T,\H}}$-valued random variables $\xi$ and $\eta$ with $\xi(0)=\eta(0)$, we have $\P$-a.s.
	\begin{equation}\label{KK}
	\begin{split} & 2\< b(t,\xi_t,\mu_t), \xi(t)\>_\H
	+\|\si(t, \xi_t,\mu_t)\|_{\Lhs_2(\U;\H)}^2 \le K\left\{\gg+\|\xi_t\|_{T,\H}^2+\mu_t(\|\cdot\|_{T, \H}^2)\right\},\\
	&2\<b(t,\xi_t,\mu_t)-b(t, \eta_t,\nu_t),\xi(t)-\eta(t)\>_\H  \le K \left\{\|\xi_t-\eta_t\|_{T, \H}^2+\W_{2,\H}(\mu_t,\nu_t)^2 \right\},\\
	& \|\si(t,\xi_t,\mu_t)-\si(t, \eta_t,\nu_t)\|_{\Lhs_2(\U;\H)}^2 \le K \left\{\|\xi_t-\eta_t\|_{T, \H}^2+\W_{2,\H}^2(\mu_t,\nu_t)^2\right\},\ \ t\in [0,T],\ \mu,\nu\in \scr P_{T,\H}^w.
	\end{split}
	\end{equation}
	Then   for any  $X_0\in L^2(\OO\to\H,\F_0,\P)$,   \eqref{E0} has a  unique  solution which is continuous in $\H$.
\end{lem}

\begin{proof}  By \eqref{KK},  the uniqueness follows from   It\^o's formula and Gr\"{o}nwall's inequality.  Below we only prove the existence by using the procedure as in \cite{W18}.

	Let $X^0(t)\equiv X_0$ and $\mu^{(0)}_t= \L_{X^0_t}.$  If for some $n\ge 1$ we have a  continuous adapted process $X^{(n-1)}(t)$ on $\H$ with $\E[\|X^{(n-1)}_T\|_{T,\H}^2]<\infty$,
	let $X^{(n)}(t)$ solve the SDE
	\begin{equation}\label{PO}
	\d X^{(n)}(t)=    b(s, X^{(n)}_s,\mu_s^{(n-1)})\d s+     \si(s, X^{(n)}_s,\mu_s^{(n-1)})\d W(s),\ \ X^{(n)}(0)=X_0,\  t\in [0,T].
	\end{equation}
	By \eqref{KK} and induction,
	 we can construct  a sequence of continuous adapted processes $\{X_T^{(n)}\}_{n\ge 1}$ on $\C_{T,\H}$ with $\sup_{n\ge1}\E[\|X_T^{(n)}\|_{T,\H}^2]<\infty$.
	Below we prove
	  that $\{X_{T}^{(n)}\}_{n\ge 1}$ is a Cauchy sequence in $L^2(\OO\to \C_{T,\H};\P)$, and hence has a limit $X_{T}$ in this space as $n\to\infty$, so that
	due to \eqref{KK} and the continuity of $b(t,\xi,\mu)$ and $\si(t,\xi,\mu)$   in $(\xi,\mu),$ we may let $n\to\infty$ in \eqref{PO} for $t\in [0,T]$ to conclude that $X_{T}$ is a solution of \eqref{E0}.
	
	By \eqref{KK} and It\^o's formula, for $Z^{(n)}(t):=X^{(n)}(t)-X^{(n-1)}(t)$,
$$\|Z^{(n)}(t)\|^2_{\H} \le K\int_0^t\left\{\|Z^{(n)}_s\|^2_{T,\H}+\E\|Z^{(n-1)}_s\|^2_{T,\H}\right\}ds+M(t) $$
where $$M(t):=2\int_0^t\left\<Z^{(n)}(s),\{\si(s,X^{(n)}_s,\mu^{(n-1)}_s)-\si(s,X^{(n-1)}_s,\mu^{(n-2)}_s)\}dW(s)\right\>_{\H}.$$\\
Then for $\lambda>0$,
\begin{equation}\label{DD1}\begin{split}
\e^{-\lambda t}\E\|Z^{(n)}_t\|^2_{T,\H}&\le K \e^{-\lambda t}\int_0^t\left\{\E\|Z^{(n)}_s\|^2_{T,\H}+\E\|Z^{(n-1)}_s\|^2_{T,\H}\right\}\d s+\e^{-\lambda t}\E\left( \sup_{0\le s \le t}M(s)\right)\\
&=:I^{(1)}(t)+I^{(2)}(t),\ \ t\in [0,T].
\end{split}
\end{equation}
We observe that
\begin{equation}\label{DD2}\begin{split}
I^{(1)}(t)&= K\int_0^t \e^{-\lambda (t-s)}\left\{\e^{-\lambda s}\E\|Z^{(n)}_s\|^2_{T,\H}+\e^{-\lambda s}\E\|Z^{(n-1)}_s\|^2_{T,\H}\right\}\d s\\
&\le \frac{K}{\lambda}\sup_{0\le s \le t}\left( \e^{-\lambda s}\E\|Z^{(n)}_s\|^2_{T,\H}\right)+ \frac{K}{\lambda}\sup_{0\le s \le t}\left( \e^{-\lambda s}\E\|Z^{(n-1)}_s\|^2_{T,\H}\right).
\end{split}
\end{equation}
By BDG's inequality, for some constants $c_1, c_2>0$, we have
\begin{equation}\label{DD3}\begin{split}
I^{(2)}(t)&\le c_1\e^{-\lambda t}
\E\left( \int_0^t \|Z^{(n)}(s)\|^2_{\H}\left\{\|Z^{(n)}_s\|^2_{T,\H}+\E\|Z^{(n-1)}_s\|^2_{T,\H}\right\}ds\right)^{\frac{1}{2}}\\
&\le c_1\e^{-\lambda t}\left( \E\|Z^{(n)}_t\|^2_{T,\H}\int_0^t \left\{\E\|Z^{(n)}_s\|^2_{T,\H}+\E\|Z^{(n-1)}_s\|^2_{T,\H}\right\}ds\right)^{\frac{1}{2}}\\
&\le \frac{1}{2}\e^{-\lambda t}\E\|Z^{(n)}_t\|^2_{T,\H}+c_2\int_0^t \e^{-\lambda (t-s)}\left\{\e^{-\lambda s}\E\|Z^{(n)}_s\|^2_{T,\H}+\e^{-\lambda s}\E\|Z^{(n-1)}_s\|^2_{T,\H}\right\}\d s\\
&\le \frac{1}{2}\e^{-\lambda t}\E\|Z^{(n)}_t\|^2_{T,\H}+\frac{c_2}{\lambda}\left\{\sup_{0\le s\le t}\left( \e^{-\lambda s}\E\|Z^{(n)}_s\|^2_{T,\H}\right)+\sup_{0\le s\le t}
\left( \e^{-\lambda s}\E\|Z^{(n-1)}_s\|^2_{T,\H}\right)\right\}.
\end{split}
\end{equation}
Substituting \eqref{DD2} and \eqref{DD3} into \eqref{DD1} yields that for $t\in[0,T]$,
\begin{equation*}
\e^{-\lambda t}\E\|Z^{(n)}_t\|^2_{T,\H}\le \frac{2(K+c_2)}{\lambda}\sup_{0\le s \le t}\left( \e^{-\lambda s}\E\|Z^{(n)}_s\|^2_{T,\H}\right)+\frac{2(K+c_2)}{\lambda}\sup_{0\le s \le t}\left( \e^{-\lambda s}\E\|Z^{(n-1)}_s\|^2_{T,\H}\right),
\end{equation*}
which  implies
\begin{equation*}
\sup_{0\le s \le T}\left( \e^{-\lambda s}\E\|Z^{(n)}_s\|^2_{T,\H}\right)\le \frac{2(K+c_2)}{\lambda}\left(\sup_{0\le s \le T}\left( \e^{-\lambda s}\E\|Z^{(n)}_s\|^2_{T,\H}\right)+\sup_{0\le s \le T}\left( \e^{-\lambda s}\E\|Z^{(n-1)}_s\|^2_{T,\H}\right)\right).
\end{equation*}
Taking $\ll= 6(K+c_2)$, we arrive at
$$
\sup_{0\le s \le T}\left( \e^{-\lambda s}\E\|Z^{(n)}_s\|^2_{T,\H}\right)\le
\frac{2(K+c_2)}{\lambda-2(K+c_2)}\sup_{0\le s \le T}\left( \e^{-\lambda s}\E\|Z^{(n-1)}_s\|^2_{T,\H}\right)= \ff 1 2 \sup_{0\le s \le T}\left( \e^{-\lambda s}\E\|Z^{(n-1)}_s\|^2_{T,\H}\right).
$$ Hence, for any $n\ge 2$ we have
$$\sup_{0\le s \le T}\left( \e^{-\lambda s}\E\|Z^{(n)}_s\|^2_{T,\H}\right)\le
\frac{1}{2^{n-1}}\sup_{0\le s \le T}\left( \e^{-\lambda s}\E\|Z^{(1)}_s\|^2_{T,\H}\right).$$
Therefore,  $\{X^{(n)}_T\}_{n\geq 1}$ is a Cauchy sequence as desired.
\end{proof}

\begin{lem} \label{L0}
Assume \ref{C1}--\ref{C3}.	
	 For any $T>0$, $X(0)\in L^2(\OO\to\H,\F_0,\P)$,     and  any $\mu\in \scr P_{T,\H}^V$, the
	SPDE
	$$\d X^\mu(t) = b(t,X_t^\mu,\mu_t)\d t + \si(t,X_t^\mu,\mu_t) \d W(t),\ \ X^\mu(0)=X_0$$
	has a unique solution $X^{\mu}_T$ satisfying
	\begin{equation}\label{Q3}
	\E\left[ V(\|X^\mu_T\|_{T,\H}^2)\right] \le  2K_1T +1+ \ff{64}{K_2}\left( K_1T+\E[V(\|X_0\|_\H^2)]\right).
	\end{equation}
\end{lem}

\begin{proof} By \ref{C1}, we see that this equation has a unique solution up to the life time $\tau.$ Now we prove that $\tau>T$ (i.e. the solution is non-explosive) and \eqref{Q3}.
	To this end, with the convention $\inf \emptyset =\infty$ we set
	\begin{align*}
	&\tau_n=\inf\{t\ge 0: \|X^\mu(t)\|_\H^2\ge n\},\ \ n\ge 1,\\
	& H(t):= \ff{\{V'(\|X^\mu(t)\|_\H^2)\|\si(t,X_t^\mu,\mu_t)^*X(t)\|_\U\}^2}{ 1+V(\|X^\mu(t)\|_\H^2)},\ \ t\in [0,T].\end{align*}
	By \ref{C3} and It\^o's formula, we obtain
	\begin{equation}\label{Q0}
	\d V( \|X^\mu(t)\|_\H^2) \le \{K_1- K_2 H(t)\} +  2 V'(\|X^\mu(t)\|_\H^2)
	\<X^\mu(t), \si(t, X_t^\mu,\mu_t)\d W(t)\>_\H.\end{equation}
	This gives rise to
	\begin{equation}\label{Q4}
	\E[V(\|X^\mu(T\land \tau_n)\|_\H^2) ] + K_2 \E \int_0^{T\land \tau_n} H(t)\d t  \le  K_1 T+ \E [V(\|X_0\|_{\H}^2)]=:C,\ \ n\ge 1.
	\end{equation}
	Then
	$$V(n) \P(\tau_n\le T) \le   \E[V(\|X^\mu(T\land \tau_n)\|_\H^2) ] \le C,\ \ n\ge 1,$$   so that by $\tau \ge \tau_n$ we obtain
	$\P(\tau\le T)\le \ff C{V(n)} \to 0$ as $n\to\infty$. Thus, $\P(\tau>T)=1.$ Moreover, by \eqref{Q0} and BDG inequality, we obtain that for all $n\geq1$,
	\begin{align*}
	\E\left[V(\|X^\mu_{T\land\tau_n}\|^2_{T,\H})\right]
	& \le K_1T + 8 \E\bigg(\int_0^{T\land\tau_n} \{V'(\|X^\mu(t)\|_\H^2)\}^2  \|\si^*(t, X_t^\mu,\mu_t)X^\mu(t)\|_{\U}^2\d t \bigg)^{\ff 1 2}\\
	&= K_1T + 8 \E\bigg(\left(1+ V(\|X^\mu_{T\land\tau_n]}\|_{T,\H}^2)\right)\int_0^{T\land\tau_n}  H(t)\d t \bigg)^{\ff 1 2} \\
	&\le K_1T + \ff 1 2 \E\left[\left(1+ V(\|X^\mu_{T\land\tau_n]}\|_{T,\H}^2)\right)\right] + 32 \E \int_0^T H(t)\d t.
	\end{align*}
	Combining this with \eqref{Q4}, we arrive at
	\beq\label{QQO} \E\left[ V(\|X^\mu_{T\land\tau_n]}\|_{T,\H}^2)\right]
	\le 2 K_1T + 1 + 64 \frac{C}{K_2}=:\dd,\ \ n\ge 1.\end{equation}
	As $C$ does not depend on $n$,   letting $n\to\infty$ and noting \eqref{Q4} give rise to \eqref{Q3}. \end{proof}

\begin{proof}[Proof of Proposition \ref{PR1}]  The  estimate \eqref{E0 estimate} is implied by Lemma \ref{L0} with $\mu_t=\L_{X_t}$ once existence
	has been established.  So, it remains to prove the existence and uniqueness.
	
 	\textbf{(a) Existence.}
	To construct a solution using  Lemma \ref{L1}, we make a localized  approximation of $b$ and $\si$ as follows. For   $\tau_n^\xi$   in \eqref{TXN}, let
	$$\phi_n(\xi)(t):= \xi(t\land\tau_n^\xi),\ \ \xi\in \C_{T,\H},\  n\ge 1, \ t\in [0,T],$$ and define
	$$ b^n(t,\xi,\mu)= b(t, \phi_n(\xi), \mu\circ\phi_n^{-1}),\ \ \si^n(t,\xi,\mu)= \si (t, \phi_n(\xi), \mu\circ\phi_n^{-1}),  \ \ \xi\in \C_{T,\H},\  \ t\in [0,T],\ \mu\in \scr P_{T,\H}.$$
	By \ref{C1}, we see that for each $n\ge 1$, $b^n$ and $\si^n$ satisfy \eqref{KK} for $\gg=1$ and some constant $K$ depending on $n$. Therefore, by Lemma \ref{L1},
	the equation
	\begin{equation}
	\label{EN}    X^n(t)= X(0)+ \int_0^t b^n(s, X_s^n, \L_{X_s^n}) \d s + \int_0^t \si^n(s, X_s^n, \L_{X_s^n}) \d W(s),\ \   t\in [0,T]
	\end{equation}
	has a unique solution. By the definition of $\phi_n$, we have
	\begin{equation}\label{tau^n}
	 \tau^n:= \inf\{t\ge 0: \|X^n(t)\|_\H^2\ge n\}= \tau_n^{X_T^n},\ \
	 \phi_n(X_s^n)= X_{s\land\tau^n}^n,\ \ s\in [0,T], n\ge 1.
	\end{equation} Moreover, for any measurable set $A\subset \C_{T,\H}$, we have
	$$\big\{(\L_{X_s^n})\circ\phi_n^{-1} \big\}(A) = \P\big(X_s^n\in \phi_n^{-1}(A)\big) = \P\big(\phi_n(X_s^n)\in A\big)= \L_{\phi_n(X_s^n)}(A)= \L_{X_{s\land\tau^n}^n}(A),$$
	so that \eqref{EN} reduces to
	\begin{equation}
	\label{EN'}
	X^n(t)= X(0)+ \int_0^t b(s, X_{s\land\tau^n}^n, \L_{X_{s\land\tau^n}^n}) \d s + \int_0^t \si(s,  X_{s\land\tau^n}^n, \L_{X_{s\land\tau^n}^n}) \d W(s),\ \   t\in [0,T].
	\end{equation} 	
So, by \ref{C3} and   applying It\^o's formula to $V(\|X^n(t)\|_\H^2)$ up to time $T\land \tau^n$, as in \eqref{QQO},  we derive
	\begin{equation}\label{W1}
	\E\left[ V(\|X^n_{T\land\tau_n]}\|_{T,\H}^2) \right] \le\dd ,\ \ n\ge 1.\end{equation}
	Consequently, the stopping times
	$$\tau_N^n:=\inf\{t\ge 0: \|X_t^n\|_{T,\H}^2\ge N\},\ \ n\ge N\ge 1$$ satisfy
	\begin{equation}
	\label{P33}
	\P(\tau_N^n<T) \le \ff {\dd} {V(N)},\ \ n\ge N\ge 1.
	\end{equation}
	Next, by \ref{C1} and \eqref{EN}, we find a constant $C_N>0$ such that  for any $n\ge N$,
	\begin{equation}
	\label{W2}
	\E \left[\sup_{s,t\in [0,T], |t-s|\le \vv} \|X^n(t\land \tau_N^n)- X^n(s\land \tau_N^n)\|_\H\right]\le C_N \vv^{\ff 1 3},\ \ 0\le s\le t\le T,\vv\in (0,T).
	\end{equation}
	Indeed, for any $l\ge 1$, by \ref{C1}, \eqref{EN}  and BDG inequality, there exists a constant $C_{N,l}>0$ such that
	$$\E\left[\sup_{t\in [s,(s+\vv)\land T]}  \|X^n(t\land \tau_N^n)- X^n(s\land \tau_N^n)\|_\H^{2l}\right] \le C_{N,l} \vv^l,\ \ n\ge N, s\in [0,T-\vv].$$
	Let  $k\in \mathbb N$  such that $k\vv\in [T,T+\vv)$. We find some constant $c(l)>0$ such that
	\begin{align*} &\E\left[ \sup_{s,t\in [0,T], |t-s|\le \vv} \|X^n(t\land \tau_N^n)- X^n(s\land \tau_N^n)\|_\H^{2l}\right]\\
	&\le c(l) \sum_{i=1}^{k} \E\left[\sup_{t\in [(i-1)\vv,\,(i\vv)\land T]}  \|X^n(t\land \tau_N^n)- X^n(\{(i-1)\vv\}\land \tau_N^n)\|_\H^{2l}\right] \le C_{N,l}(T+\vv)\vv^{l-1},\ \ n\ge N.\end{align*}
	Therefore, by Jensen's inequality, we obtain
	$$\E\left[ \sup_{s,t\in [0,T], |t-s|\le \vv} \|X^n(t\land \tau_N^n)- X^n(s\land \tau_N^n)\|_\H\right]\le
	\left\{ C_{N,l}(T+\vv)\right\}^{\ff 1 {2 l}} \vv^{\ff 1 2-\ff 1 {2l}}, \ \ n\ge N.$$
	Taking $l\ge 1$ such that $\ff 1 2-\ff 1 {2l}\ge \ff 1 3$, we obtain \eqref{W2}.
	Particularly, \eqref{W2} holds true for $n=N$. In this case, $\tau_N^n=\tau_n^n=\tau^n$.
	Due to this and \eqref{W1},  and noting that embedding $\H\hookrightarrow\BB$ is compact, we deduce from  the Arzel\'a-Ascoli type theorem  for measures  that $\{\mu^n:= \L_{X^n_{T \wedge \tau^n}}\}_{n\ge 1}$
	is tight in $\scr P_{T,\BB}$.   By the Prokhorov theorem, for some  subsequence $\{n_k\}_{k\ge 1}$
	we have  $ \mu^{n_k}\to \mu$    weakly  in $\scr P_{T,\BB}$   as $k\to\infty.$    Notice that $\phi_{n}(\xi)=\xi$ for $\xi\in\C_{T,\H,n}$ and define $$\tau_N^{k,j}:=\tau_N^{n_k}\land \tau_N^{n_j}.$$ Then we find
	$$\phi_{n_i} (X_{t\land \tau_N^{k,l}}^{n_j})= X_{t\land \tau_N^{k,l}}^{n_j},\ \ i,j\in \{k,l\},$$ and
		$$ \lim_{k\to\infty}\lim_{l\to\infty}\mu^{n_k}\circ\phi_{n_l}^{-1} = \mu\ \text{weakly  in}\ \scr P_{T,\BB}.$$
		
From the above two properties, \eqref{W1} and \ref{C1},
we find a family of constants $\{\vv_{k,l}: k,l\ge 1\}$ with
	$\vv_{k,l}\to 0$ as $k,l\to \infty$ such that
		\begin{equation}\label{CVV-b} \begin{split}
	&\left\|b\left(t, X^{n_k}_{t\land \tau_N^{k,l}}, \mu_t^{n_k} \right)
	-b\left(t, X_{t\land \tau_N^{k,l}}^{n_l}, \mu_t^{n_l} \right)\right\|_\H \\
	=&\left\|b\left(t, X^{n_k}_{t\land \tau_N^{k,l}}, \mu_t^{n_k}\circ\phi_{n_l}^{-1} \right)
	-b\left(t, X_{t\land \tau_N^{k,l}}^{n_l}, \mu_t^{n_k}\circ\phi_{n_l}^{-1} \right)\right\|_\H \\
	&+\left\|b\left(t, X^{n_l}_{t\land \tau_N^{k,l}}, \mu_t^{n_k}\circ\phi_{n_l}^{-1}\right)
	-b\left(t, X_{t\land \tau_N^{k,l}}^{n_l}, \mu_t^{n_l}\circ\phi_{n_l}^{-1} \right)\right\|_\H \\
	\leq& C_N\big\|X^{n_k}_{s\land \tau_N^{k,l}}
	-X^{n_l}_{s\land \tau_N^{k,l}}\big\|_{T,\H}+C_N\vv_{k,l},\ \ \ l\ge k\ge N.
	\end{split}\end{equation}
Similarly, we also have
	\begin{equation}\label{CVV-sigma}
	\big\|\si(t, X^{n_k}_{t\land \tau_N^{k,l}}, \mu_t^{n_k} )- \si(t, X_{t\land \tau_N^{k,l}}^{n_k}, \mu_t^{n_l} )\big\|_{\Lhs_2(\U;\H)}
	\leq  C_N\big\|X^{n_k}_{s\land \tau_N^{k,l}}
	-X^{n_l}_{s\land \tau_N^{k,l}}\big\|_{T,\H}+C_N\vv_{k,l},\ \ \ l\ge k\ge N.
	\end{equation}
	By \eqref{CVV-b}, \eqref{CVV-sigma},  \ref{C1}, and applying BDG inequality, we find  a constant $C_N>0$   such that
	\begin{align*}\E\left[
	\big\|X^{n_k}_{t\land \tau_N^{k,l}}-X^{n_l}_{t\land \tau_N^{k,l}}\big\|^2_{T,\H}\right]
	\le C_N^2 \int_0^T  \E\left[
	\big\|X^{n_k}_{s\land \tau_N^{k,l}}-X^{n_l}_{s\land \tau_N^{k,l}}\big\|^2_{T,\H}\right]\d s
	+ C_N^2\vv^2_{k,l}T,\ \ t\in [0,T],\ l\ge k\ge N.
	\end{align*}
	Applying Gr\"{o}nwall's inequality with noting that $\vv_{k,l}\to 0$ as  $k,l\to\infty$, we derive
	\begin{equation}\label{A33}
	\lim_{k\to\infty} \sup_{l\ge k}\ \E\left[
	\big\|X^{n_k}_{T\land \tau_N^{k,l}}-X^{n_l}_{T\land \tau_N^{k,l}}\big\|^2_{T,\H}\right]
	\le
		C_N^2\lim_{k\to\infty}\sup_{l\ge k}\ \vv^2_{k,l}T\e^{C_N^2T} =0.
	\end{equation}
	Then we infer from \eqref{P33} that for any $\epsilon>0$,
	\begin{align*}
	&\P\left(   \|X_T^{n_k}-X_T^{n_l}  \|_{T,\H}>\epsilon\right) \\
	\le &\P(\tau^{n_k}_{N}\le T)+ \P(\tau^{n_l}_{N}\le T)
	+ \P\left( \|X_{T\land \tau_{N}^{k,l}} ^{n_k}-X_{T\land \tau_{N}^{k,l}}^{n_l}  \|_{T,\H}>\epsilon\right)\\
	\le &\ff{2\dd}{V(N)}+ \P\left( \|X_{T\land \tau_N^{k,l}}^{n_k}-X_{T\land \tau_N^{k,l}}^{n_l}  \|_{T,\H}>\epsilon\right),\ \ l\ge k\ge N.
	\end{align*}
	Combining this with \eqref{A33}, we  obtain
	$$	\lim_{k\to\infty} \sup_{l\ge k}\P\left(  \|X^{n_k}_T-X^{n_l}_T\|_{T,\H}>\epsilon\right)\le \ff{2\dd}{V(N)},\ \ N\ge 1,\ \epsilon>0.$$
	Letting $N\to\infty$,
	we conclude  that $X_T^{n_k}$ converges in probability to some ${\C_{T,\H}}$-valued  random variable $X_T$. Since for each $n\ge1$, $X_T^n$ is adapted, so is $X_T$.
	Therefore, up to a subsequence $\{\tt n_k\}_{k\ge 1}$, we have $\P$-a.s.
	$$\lim_{n\to\infty}   \|X^{\tt n_k}_T -X_T\|_{T,\H } =0.$$   In particular, $\L_{X_T^{\tt n_k}}\to \L_{X_T}$ weakly in $ \scr P_{T,\H}$, and
$$\tau_N':=\inf\Big\{\sup_{k\ge 1} \|X^{\tt n_k}(t)\|_\H\ge N\Big\}\uparrow\infty\ \text{as}\ N\uparrow\infty.$$
 Since
	$\mu^{\tt n_k}\to \mu$ weakly in $\scr P_{T,\BB}\supset \scr P_{T,\H}$, as is proved above, we have $\L_{X_T}=\mu$.
	Combining this with   \ref{C1}, \ref{C2}   and \eqref{W1},  we may let $k\to\infty$ in \eqref{EN} for $n=\tt n_k$ to conclude that
	$X_T$  satisfies 
$$\<X(t\land\tau_N'), \eta\>_\H= \<X(0), \eta\>_\H+\int_0^{t\land\tau_N'} \<b(s, X_s,\mu_s),\eta\>_\H\d s + \int_0^{t\land\tau_N'} \<\si(s, X_s,\mu_s)\d W(s),\eta\>_\H $$ for any $ t\in [0,T], N\ge 1$ and $ \eta\in \H_0.$
Since $\H_0 $ is dense in $\H$ and $\tau_N'\uparrow\infty$ as $N\uparrow\infty$, this implies that $X_T$
 solves \eqref{E0}.
	
	\textbf{ (b) Uniqueness.}  If $C_N$ is bounded, by letting $N\to\infty $ in \ref{C4} we find a global Lipschitz  condition on the coefficients which, as is well known,  implies the pathwise uniqueness.  So, below we assume $C_N\to\infty$ as $N\to\infty$.

\textbf{(b1)} We first  prove the pathwise uniqueness up to a time $t_0\in (0,T]$.
	    Let $X_T$ and $Y_T$ be two solutions with $X(0)=Y(0)$.  Let
	\begin{align}
	\tau_n=\tau_n^X\land\tau_n^Y= \inf\{t\ge 0: \|X(t)\|_\H\lor\|Y(t)\|_{\H}\ge n\},\ \ n\ge 1.\label{X Y tau}
	\end{align}
	Then $Z_T=X_T-Y_T$ satisfies
	\begin{align*}
	Z(t\wedge \tau_n) =  &\int_0^{t\wedge \tau_n}\left(b(t,X_t,\L_{X_t})-b(t,Y_t,\L_{Y_t})\right)\d t \\
	&+\int_0^{t\wedge \tau_n} \left(\si(t,X_t,\L_{X_t})- \si(t,Y_t,\L_{Y_t})\right) \d W(t)
	\end{align*}
	By  It\^o's formula and BDG's inequality, there exist   constants $c_1,c_2>0$ such that
	\begin{equation}\label{ADD} \beg{split}
	\E\|Z_{\tau_n\wedge s}\|^2_{T,\BB}
	\leq& c_1 \E\int_0^{\tau_n\wedge s}\|b(t,X_t,\L_{X_t})-b(t,Y_t,\L_{Y_t})\|_{\BB}\|Z(t)\|_{\BB}\d t  \\
	&+c_1 \E\left(\int_0^{\tau_n\wedge s}\|\si(t,X_t,\L_{X_t})-\si(t,Y_t,\L_{Y_t})\|^2_{\Lhs_2(\U;\BB)}\|Z(t)\|^2_{\BB}\d t\right)^{\frac12} \\
	&+c_1 \E\int_0^{\tau_n\wedge s}\left\|\left(\si(t,X_t,\L_{X_t})- \si(t,Y_t,\L_{Y_t})\right) \right\|^2_{\Lhs_2(\U;\BB)}\d t \\
	\leq& \frac{1}{2}\E\|Z_{\tau_n\wedge s}\|^2_{T,\BB}+
	c_2\E\int_0^{\tau_n\wedge s}\|b(t,X_t,\L_{X_t})-b(t,Y_t,\L_{Y_t})\|^2_{\BB}\d t  \\
	&
	+c_2\E\int_0^{\tau_n\wedge s}\left\|\left(\si(t,X_t,\L_{X_t})- \si(t,Y_t,\L_{Y_t})\right) \right\|^2_{\Lhs_2(\U;\BB)}\d t,\ \ s\in [0,T]. \end{split}
	\end{equation}
Since $\pi_t:=\L_{(X_t,Y_t)}\in{\C(\L_{X_t},\L_{Y_t})}$ and by the definition of $\W_{2,  \BB,n}$ in \eqref{local Wasserstein}, we have
\begin{equation}\label{ADD1}
\W_{2, \BB, n}(\L_{X_t},\L_{Y_t})^2\le \int_{\C_{T,\BB}\times\C_{T,\BB}}   \|\xi_{ t \wedge \tau_n^{\xi} \wedge\tau_n^{\eta} }-\eta_{ t \wedge \tau_n^{\xi} \wedge\tau_n^{\eta} }\|^2_{T,\BB}\pi(d\xi, d\eta)
=\E \|X_{\tau_n\wedge t}-Y_{\tau_n\wedge t}\|^2_{T,\BB}.
\end{equation}
So,   by \ref{C4},  we have
\begin{equation}\label{ADD2}
\begin{split}
&\E\int_0^{\tau_n\wedge s}\big\{\|b(t,X_t,\L_{X_t})-b(t,Y_t,\L_{Y_t})\|_{\BB}^2+ \left\|\left(\si(t,X_t,\L_{X_t})- \si(t,Y_t,\L_{Y_t})\right) \right\|^2_{\Lhs_2(\U;\BB)}\big\}\d t\\
&\le C_n\E\int_0^{\tau_n\wedge s}\left[\|X_t-Y_t\|_{T, \BB}^2+\W_{2, n, \BB}(\L_{X_t},\L_{Y_t})^2+C_0\e^{-C_n\ee}\right]\d t\\
&\le C_n\int_0^{s}\left[\E\|Z_{\tau_n\wedge s}\|_{T, \BB}^2+\W_{2, \BB, n}(\L_{X_t},\L_{Y_t})^2+C_0\e^{-C_n\ee}\right]\d t\\
&\le 2C_n\int_0^{s}\E\|Z_{\tau_n\wedge s}\|_{T, \BB}^2\d t+C_nC_0\e^{-C_n\ee},
\end{split}
\end{equation}
which together with \eqref{ADD} yields
\begin{equation}\label{A*D}
	\E\left[\|Z_{\tau_n\wedge s}\|^2_{T,\BB} \right]
	\leq C
	C_n\int_0^{s}\left\{\E\|Z_{\tau_n\wedge t}\|^2_{\BB}+ C_0 \e^{-\vv C_n} \right\} \d t,\ \ n\ge 1
	\end{equation} for some constant $C>0$.
	Applying Fatou's lemma  and  Gr\"{o}nwall's inequality, we derive
	$$\E \|Z_s\|_{T,\BB}^2\le \liminf_{n\to\infty} \E\left[\|Z_{\tau_n\wedge s}\|^2_{T,\BB} \right]\le  sCC_0 \liminf_{n\to\infty}C_n \e^{-C_n(\vv-Cs)}=0,\ \  T\ge s\in (0, \vv/C).$$
This implies the pathwise uniqueness up to time $t_0:= \{\vv/C\}\land T.$

\textbf{(b2)} If $t_0=T$, then the proof is finished. Otherwise, since $Z_{t_0}=0$, \eqref{A*D} implies
$$\E\left[\|Z_{\tau_n\wedge s}\|^2_{T,\BB} \right]
	\leq C
	C_n\int_{t_0}^{s} \E\|Z_{\tau_n\wedge t}\|^2_{\BB} \d t + sC_0 \e^{-\vv C_n},\ \ n\ge 1,\ s\in [t_0,T].  $$
Using Fatou's lemma  and  Gr\"{o}nwall's inequality as before, we arrive at
$$\E \|Z_s\|_{T,\BB}^2\le \liminf_{n\to\infty} \E\left[\|Z_{\tau_n\wedge s}\|^2_{T,\BB} \right]\le  sCC_0 \liminf_{n\to\infty}C_n \e^{-C_n(\vv-C(s-t_0))}=0,\ \    T\ge s\in (t_0, t_0+\vv/C).$$
Thus, the uniqueness holds up to time $(2t_0) \land T$. Repeating the procedure for finite many times, we prove the uniqueness up to time $T$.
\end{proof}

\section{Proof of Theorem \ref{T1} } \label{Section:Proof of T1}

\beg{proof}[Proof of \ref{Thm-Existence-weak} in Theorem \ref{T1}] 	
For each $n\ge 1$, let
$$b_{n}(t,\xi,\mu):= B_n(t, \xi(t)) + b(t,\xi_t,\mu_t),\ \ (t,\xi,\mu)\in [0,T]\times\C_{T,\H}\times\scr P_{T,\H}.$$  Obviously,  \ref{A1}--\ref{A3}  imply \ref{C1}--\ref{C3} for  $(b_n,\si)$ replacing $(b,\si)$.
Thus, by Proposition \ref{PR1},   there exists a  continuous adapted process $X^n(t)$ on $\H$ such that
\begin{equation}\label{P11}
\begin{split}
X^n(t)= X(0) &+ \int_0^t \left\{B_n(s,  X^n(s)) + b(s,X^n_s,\L_{X_s^n} )\right\}\d s\\
&+\int_0^t \si(s,X_s^n,\L_{X_s^n})\d W(s), \ \  t\in [0,T],
\end{split}
\end{equation}
and
\begin{equation} \label{P20}
\E\left[V(\|X^n_T\|_{T,\H}^2)\right]
\le \dd := 2K_1T +1+ \ff{64}{K_2}\left( K_1T+\E[V(\|X_0\|_\H^2)]\right),\ \ n\ge 1.
 \end{equation}
Consequently,
the stopping times
$$\tau_N^n:=\inf\{t\ge 0: \|X_t^n\|_{T,\H}^2\ge N\},\ \ n,N\ge 1$$ satisfy
\begin{equation}\label{P3'} \P(\tau_N^n<T) \le \ff {\dd} {V(N)},\ \ n,N\ge 1.\end{equation}
Next, similarly to \eqref{W2}, by \ref{A1}, the first inequality in \ref{A2}, \eqref{P11} and noting that $\|\cdot\|_\BB\le c\|\cdot\|_\H$ for some constant $c>0$, we find  a constant $C_N>0$ such that
\begin{equation}\label{W2'} \E \left[\sup_{s,t\in [0,T], |t-s|\le \vv} \|X^n(t\land \tau_N^n)- X^n(s\land \tau_N^n)\|_\BB\right]\le C_N \vv^{\ff 1 3},\ \ 0\le s\le t\le T,\ \vv\in (0,T).\end{equation}
		Now, combining \eqref{W2'}  with \eqref{P3'}, we arrive at
	\begin{align*} & \E \left[\sup_{s,t\le T, |s-t|\le \vv} (1\land \|X^n(s)-X^n(t)\|_{\BB})\right]\\
	&\le  \P(\tau^n_N\le T) +  \E \left[\sup_{s,t\le T\land \tau_N^n, |s-t|\le \vv} (1\land \|X^n(s)-X^n(t)\|_{\BB})\right]\\
	&\le \ff{\dd}{V(N)} + C_N  \vv^{\ff 1 3},\ \ n,  N\ge 1,\vv>0.\end{align*}
	Since $V(N)\uparrow \infty$ as $N\uparrow \infty$, we obtain
\begin{equation}\label{P00'} \E \left[\sup_{s,t\le T, |s-t|\le \vv} \left(1\land \|X^n(s)-X^n(t)\|_\BB\right)\right]\le \inf_{N>0}\left\{ \ff{\dd_{T,X(0)}}{V(N)}+ C_N \vv^{\ff 1 3}\right\}\downarrow 0\text{\ as\ }\vv\downarrow 0.\end{equation}
Due to this and \eqref{P20}, one can use  the Arzel\'a-Ascoli theorem for measures to find that $\{\mu^n:= \L_{X_T^n}\}_{n\ge 1}$ is tight in $\scr P_{T,\BB}$, so is $\{\LL^n:=\L_{(X_T^n,Y_T^n, W_T)}\}_{n\ge 1}$, where $W_T$ is a continuous process on a separable Hilbert space ${\tt{\mathbb U}}$ such that the embedding $\mathbb U\subset \tt{\mathbb U}$ is Hilbert-Schmidt, and
$$Y^n(t):=\int_0^t \si(s,X_s^n,\mu_s^n)\d W(s),\ \ t\in [0,T]$$
is a continuous process on $\BB$.
By the Prokhorov theorem, there exists a subsequence $\{n_k\}_{k\ge 1}$
such that   $ \mu^{(n_k)}\to \mu$    weakly  in $\scr P_{T,\BB}$,    and  $\LL^{n_k} \to \LL$ weakly in the probability space on $\scr P(\C_{T,\BB}^2\times \tt{\mathbb U})$. Then the Skorokhod theorem guarantees that there exists a complete filtration probability space
$(\tt\OO,\{\tt\F_t\}_{t\ge 0}, \tt\P)$ and a sequence $(\tt X_T^{n_k}, \tt Y_T^{n_k}, \tt W_T^{n_k})$ such that $\LL^{n_k}= \L_{(\tt X_T^{n_k},\tt Y_T^{n_k}, \tt W^{n_k}_T)|\tt\P}$ and
\beq\label{CCV}  \lim_{k\to\infty}\left(  \|\tt X_T^{n_k}-\tt X_T\|_{T,\BB}+\|\tt Y_T^{n_k}- \tt Y_T\|_{T,\BB} \right)=0
\end{equation} holds for some
continuous adapted process $(\tt X_T, \tt Y_T)$ on $\BB$.  Since the embedding $\H\hookrightarrow\BB$ is continuous, there exist    continuous maps $\pi_m: \BB\to \H,\ m\ge 1$ such that
$$ \|\pi_mx\|_\H\le \|x\|_\H,\ \  \lim_{m\to\infty} \|\pi_m x\|_\H = \|x\|_\H,\ \ x\in \BB,$$  where $\|x\|_\H:=\infty$ if $x\notin\H.$
Recalling  $\L_{\tt X_T^{n_k}|\tt\P}= \L_{X_T^{n_k}|\P}$, $\tt X_T^{n_k}\to \tt X_T$ in $\C_{T,\BB}$ as $k\to\infty$, \eqref{P20}  and  Fatou's lemma,   one has
\begin{equation}\label{QN}
\begin{split} &\tt \E \left[V(\|\tt X_T\|_{T,\H}^2) \right]
\leq
\tt \E\left[\lim_{m\to\infty} V(\|\pi_m \tt X_T\|_{T,\H}^2) \right]
\le \liminf_{m\to\infty}
\tt \E\left[V(\|\pi_m \tt X_T\|_{T,\H}^2) \right]\\
&= \liminf_{m\to\infty} \liminf_{k\to\infty}\tt \E\left[V(\|\pi_m \tt X_T^{n_k}\|_{T,\H}^2) \right]\le \dd <\infty.
\end{split}\end{equation}
Therefore we can infer from $\L_{\tt X_T^{n_k}|\tt\P}= \L_{X_T^{n_k}|\P}$, \eqref{P20} and \eqref{QN} that $\tt \P$-a.s.,
\begin{equation}\label{STP}
\tt\tau_N:= \inf\left\{t\ge 0: \sup_{k\ge 1} \|\tt X^{n_k} (t)\|_{\H}\ge N\right\}\uparrow \infty\ \text{as}\ N\uparrow\infty.
\end{equation}
Since $\tt Y^{n_k}_T$ is a continuous local martingale on $\BB$ with quadratic  variational process
$$\<\tt Y^{n_k}\>(t)=\int_0^{t} (\si^*\si)\left(s,\tt X_s^{n_k},\mu_s^{n_k}\right)\d s,\ \ t\in [0,T]. $$
We deduce from \eqref{P20},  \eqref{CCV}, \eqref{STP}  and   \ref{A1}  that  $\tt Y_T$ is a continuous local martingale on $\BB$ with quadratic variational process
$$\<\tt Y\>(t)= \int_0^t (\si^*\si)\left(s,\tt X_s,\L_{\tt X_s|\tt \P}\right)\d s,\ \ t\in [0,T].$$
By the martingale representation theorem,    there exists a cylindrical Brownian motion $\tt W(t)$ on $\mathbb U$ under $\tt\P$ such that
\beq\label{DCC} \tt Y(t)= \int_0^t \si\left(s,\tt X_s,\L_{\tt X_s|\tt \P}\right)\d \tt W(s),\ \ t\in [0,T].\end{equation}
 Moreover, it follows from \eqref{P11} and $ \L_{(\tt X_T^{n_k},\tt W^{n_k}_T)|\tt\P}=  \L_{(X_T^{n_k},W_T)|\P}$ that $\tt\P$-a.s.,
 \beg{align} \label{Xnk equ}
\tt X^{n_k}(t)= \tt X^{n_k}(0) +
\int_0^t \left\{B_{n_k}\left(s, \tt X^{n_k}(s)\right) + b\left(s,\tt X^{n_k}_s,\mu_s^{n_k}\right)\right\}\d s+ Y^{n_k}(t), \ \  t\in [0,T],k\ge 1.
\end{align}
 So, for any $N,k\ge 1$,
\beg{align*} \tt X^{n_k}(t\land\tt\tau_N)= \tt X^{n_k}(0) &+ \int_0^{t\land\tt\tau_N} \left\{B_{n_k}\left(s, \tt X^{n_k}(s)\right) + b\left(s,\tt X^{n_k}_s,\mu_s^{n_k}\right)\right\}\d s +Y^{n_k}(t \land\tt\tau_N), \ \  t\in [0,T].
\end{align*}
Summarizing this,  \ref{A1}, \ref{A2}, \eqref{P20}, \eqref{CCV} and \eqref{DCC}, and then letting $k\to\infty$, we derive
\begin{align}
\tensor[_\BB]{\left\<\tt X(t\land\tt\tau_N),\eta\right\>}{_{\BB^*}}
=&\tensor[_\BB]{\left\<\tt X(0),\eta\right\>}{_{\BB^*}}
+ \int_0^{t\land\tt\tau_N} \left\{
\tensor[_\BB]{\left\<B(s,  \tt X) + b\left(s,\tt X_s,\L_{\tt X_s|\tt\P}\right), \eta\right\>}{_{\BB^*}}\right\}\d s\notag\\
&+ \tensor[_\BB]{\left\<\int_0^{t\land\tt\tau_N}\si\left(s,\tt X_s,\L_{\tt X_s|\tt\P}\right) \d\tt W(s),\eta \right\>}{_{\BB^*}}, \ \ \eta\in \H_0.\label{WS}
\end{align}
 It is easy to see that \ref{A1}, \ref{A2}   and \eqref{QN}  imply  that  for some constant $\tt C_N>0$,
$$\sup_{s\in [0, T\land\tt\tau_N] }\|\si(s,\tt X_s,\L_{\tt X_s|\tt\P})\|_{\Lhs_2(\mathbb U;\H)}\le \tt C_N,$$  which means
 $\int_0^{t \land\tt\tau_N}\si(s,\tt X_s,\L_{\tt X_s|\tt\P})\d \tt W(s)$ is an adapted continuous process   on $\H\subset \BB$. Similarly, by \ref{A1}, \ref{A2} and \eqref{QN},
 \begin{equation}\label{B+b conti}
\int_0^{t\land\tt\tau_N} \{ B(s,  \tt X) + b(s,\tt X_s,\L_{\tt X_s|\tt\P}  ) \}\d s
 \end{equation}
 is a continuous process on $\BB$ as well.  On account of \eqref{QN} and \eqref{STP}, we identify that $(\tt X_T,\tt W_T)$ is a weak solution of \eqref{EE}.
 \end{proof}

\beg{proof}[Proof of \ref{Thm-Coninuous} in Theorem \ref{T1}] 	  Now, assume \ref{A4}.  We aim to prove the continuity of $\tt X(t)$ in $\H$. Since $X(t)$ is an adapted  continuous process on $\BB$, and   hence weak continuous in $\H$, it suffices to  prove the continuity of $[0,T]\ni t\mapsto \|\tt X(t)\|_\H$.
By \eqref{STP}, we only need to prove the continuity up to time $\tt\tau_N$ for each $N\ge 1$,  where $\tau_N$ is given in  \eqref{STP}.  If $\tt X\in\H$, then $B(t,\tt X)\in\BB$ and $\<B(t,\tt X),\tt X\>_\H$ does not make sense, therefore we can not use the  It\^{o} formula to $\|\tt X\|^2_{\H}$ directly. To overcome this difficulty, we consider $\|T_m\tt X\|^2_{\H}$ firstly, where $T_m$ is the operator as in \ref{A4}.  Applying $T_m$ to \eqref{EE} with noting  \ref{A4},  we see that
\begin{align} T_m \tt X(t\land\tt\tau_N) = &T_m(\tt X(0)) +\int_0^{t\land\tt\tau_N} T_m \left\{B(r,\tt X(r))+ b(r, \tt X_r, \L_{\tt X_r|\tt\P})\right\}\d r \notag\\
&+\int_0^{t\land\tt\tau_N} T_m \si(r, \tt X_r, \L_{\tt X_r|\tt\P})\d W(r),\ \ t\in [0,T]
\label{mollify Ito}
\end{align}
is an
$L^p$-semimartingale on $\H$ for any $p  \in [1,\infty)$.

%
%
%

Combining this with \ref{A1}, \ref{A4} and the   It\^o's formula, we find a constant $C_N>0$  such that
$$\tt\E\left[\left( \|T_m \tt X(t\land\tt\tau_N)\|_{\H}^2 - \|T_m \tt X(s\land\tt\tau_N)\|_{\H}^2\right)^{4}\right] \le C_N (t-s)^2,\ \ [s,t]\subset[0,T],\ t-s<1,\ m\ge 1.$$
Since $\|T_m x-x\|_\H\to 0$ as $m\to\infty$ holds for $x\in\H$ and $\tt X(t)$ takes values in $\H$,  Fatou's lemma   implies
$$\tt\E\left[\left( \|\tt X(t\land\tt\tau_N)\|_{\H}^2 -  \|\tt X(s\land\tt\tau_N)\|_{\H}^2\right)^{4}\right] \le C_N (t-s)^2,\ \ [s,t]\subset[0,T],\ t-s<1.$$
Therefore,
Kolmogorov's continuity theorem ensures the continuity of $t\mapsto \|\tt X(t\land\tt\tau_N)\|_\H$ as desired.\end{proof}

\beg{proof}[Proof of \ref{Thm-Unique} in Theorem \ref{T1}] 	By \ref{Thm-Existence-weak} in Theorem \ref{T1}, \eqref{EE} has a weak solution. Moreover, for any fixed $\mu\in \scr P_{T,\H}^w$,
it is easy to deduce from \ref{A1}, \ref{A2}, \ref{A3} and  \ref{A5} that the distribution independent SPDE
$$\d X^\mu(t)= \left\{B(t,X^\mu(t))+b(t, X_t^\mu, \mu_t)\right\}\d t +\si(t,X_t^\mu,\mu_t)\d W_t,\ \ X^\mu(0)=X_0$$ has a unique solution. So, by a  Yamada-Watanabe type principle, see for instance \cite[Lemma 3.4]{HW19} and \cite{K-14-ECP}, it remains to prove the pathwise uniqueness.

As is explained in step \textbf{(b2)} in the proof of Proposition \ref{PR1}, we assume that $C_N\to\infty $ as $N\to\infty$ and  it suffices to prove the pathwise uniqueness up to a time $t_0>0$ independent of the initial value $X(0)$.   Let $\tau_n$ be defined by \eqref{X Y tau}.  As is shown in \textbf{(b1)} in the proof of Proposition \ref{PR1},   it follows from \ref{A5},   It\^o's formula and BDG inequality that there is a constant $K_0>1$   such that
	 $$\E\left[\|Z_{\tau_n\wedge s}\|^2_{T,\BB}\right] \le     K_0C_n
	 \int_0^{s}\left( \E\left[\|Z_{\tau_n\wedge r}\|^2_{T,\BB}\right]  + \e^{-\vv C_n}\right)\d r,\ \ s\in [0,T], n\ge 1.$$
	By Fatou's lemma and Gr\"onwall's inequality, this implies
$$\E\left[ \|Z_s\|^2_{T,\BB}\right] \le \liminf_{n\to\infty} \E\left[\|Z_{\tau_n\wedge s}\|^2_{T,\BB}\right]\le \liminf_{n\to\infty} sK_0\e^{ K_0C_ns-\vv C_n} =0$$
provided $s<t_0:= \vv/K_0.$ Therefore pathwise uniqueness holds up to time $t_0$, and hence the proof is finished.  \end{proof}

\section{Proof of Theorem \ref{T-TE}} \label{Section:Proof of T2}

 It suffices to verify conditions in Theorem \ref{T1} for suitable choices of $\H,\BB, B_n$, $J_n$ and $T_n$.
Let $j(x)$ be a Schwartz function such that $0\leq\widehat{j}(\xi)\leq1$ for all the $\xi\in \R^d$ and $\widehat{j}(\xi)=1$ for any $|\xi|\leq1$.
For any $n\ge 1$ and $f\in H^0:=L^2(\mathbb T^d\to\R^d;\mu)$,  we define
 \begin{align}\label{JN}
 J_n f:=j_n*f,\ \ j_{n}(x)=\frac{1}{2\pi}\sum_{k\in{\Z}^d}\widehat{j}\left(k/n\right)\e^{{\rm i}\<k,\cdot\>},
 \end{align}
and
 \begin{align}\label{TN}
 T_n f:=(I-n^{-2} \Delta)^{-1}f= \sum_{k\in\Z^d}\left( 1+n^{-2} |k|^2\right)^{-1}  \widehat{f}(k)\, \e^{{\rm i}\<k,\cdot\>}.
\end{align}

Obviously, for any $s\ge 0$,
\begin{align}
D^sJ_{n}&=J_{n}D^s,\ \ D^s T_n= T_n D^s, \label{Jn property 3}\\
\<J_{n}f, g\>_{H^s}&=\<f, J_{n}g\>_{H^s},\ \ \<T_nf,g\>_{H^s} = \<f, T_ng\>_{H^s}, \ \ f,g\in  H^s, \label{Jn property 4}\\
\|J_{n}f\|_{H^s}\lor \|T_{n}f\|_{H^s}&\leq \|f\|_{H^s},\ \|\nn J_nf\|_{H^s}\lor \|\nn T_nf\|_{H^s}\lesssim n\|f\|_{H^s},\ \ n\ge 1, f\in H^s,\label{Jn property 5}
\end{align}
 where  for two sequences of positive numbers $\{a_n,b_n\}_{n\ge 1}$,    $a_n \lesssim b_n$ means that
$a_n\le c b_n$ holds for some constant $c>0$ and all $n\ge 1.$ Moreover,  we write $a_n={\rm o}(b_n)$ if $\lim_{n\to\infty } b_n^{-1} a_n =0.$  Then
\begin{equation}
\|X-J_{n}X\|_{H^r}={\rm o } ( n^{r-s}),\ \ 0\le  r\leq s, X\in H^s,\label{Jn property 1}
\end{equation}
  and  for any $r\ge s$,
\begin{equation}
\|J_{n}X\|_{H^{r}}\lesssim  n^{r-s}\|X\|_{H^{s}} \
 \text{uniformly\ in}\  X\in H^s.  \label{Jn property 2}
\end{equation}
To verify conditions in Theorem \ref{T1},
we need more properties of $J_n$, $T_n$ and $D^s$.  In general, the  commutator  for two operators $P,Q$ is given by
  $$[P,Q]:= PQ-QP.$$

\begin{lem}\label{Je commutator}  There exists a constant $C>0$ such that  	\begin{align*}
	\|[T_{n}, (g\cdot \nabla)]f\|_{L^2(\mu)}
	\leq C\|\nabla g\|_{\infty}\|f\|_{L^2(\mu)},\ \ n\ge 1, f\in L^2(\mathbb T^d\to\R^d;\mu),\  g\in W^{1,\infty}(\mathbb{T}^d\to\R^d;\mu).
	\end{align*}
\end{lem}
\begin{proof}
	
	It is worth noticing that in 1-D case, the above commutator
	estimate has been established for a different mollifier on
	the whole space, see \cite{HK-09-DIE}. In current setting, periodicity is required and the mollifier is different, so we present also the proof here. 	
	
	 Let $\pp_l$ denote the $l$-th partial derivative in $\R^d$. Since $[T_n,\partial_l]=0$ for $l\in\{1,2,\cdots,d\}$, we have
	\begin{align*}
	&\|[T_{n}, (g\cdot \nabla)]f\|^2_{L^2(\mu)}
	=\sum_{j=1}^d
	\Big\|\sum_{l=1}^dT_{n}\left(g_l\partial_{l}f_j\right)
	-\sum_{l=1}^dg_l\partial_{l}\left(T_{n}f_j\right)\Big\|^2_{L^2(\mu)}\\
	\leq& d\sum_{j,l=1}^d \big\|T_{n}\left(g_l\partial_{l}f_j\right)
	-g_lT_{n}\left(\partial_{l}f_j\right)\big\|^2_{L^2(\mu)}
	=d\sum_{j,l=1}^d \Big\|[T_{n},g_l]\partial_{l}f_j\Big\|^2_{L^2(\mu)}.
	\end{align*}
	 Hence,
	it suffices to find a constant $c>0$ such that 	\begin{equation}\label{Te commutator xl}
	\|[T_{n}, g]\partial_{l}f\|^2_{L^2(\mu)}\leq c\|\nn g\|^2_{L^{\infty}}\|f\|^2_{L^2(\mu)}, \ \ f,g\in C^1(\mathbb T^d), 1\le l\le d, n\ge 1.
	\end{equation}
Noting that
	 $$ \ff{1}{1+\frac{1}{n^2} |k|^2}- \ff{1}{1+\frac{1}{n^2} |m|^2} =\ff{\<m-k,m+k\>}{n^2 (1+\frac{1}{n^2} |k|^2)(1+\frac{1}{n^2} |m|^2)}
	 = \sum_{j=1}^d \ff{(m-k)_j (m_j+k_j)}{n^2 (1+\frac{1}{n^2} |k|^2)(1+\frac{1}{n^2} |m|^2)},$$
	 by  $T_n=(I-\ff 1 n \DD)^{-1}$,     \eqref{X0},  and \eqref{XXX},  we find a constant $c>0$ such that
	\begin{align*}
	&\|[T_{n}, g]\partial_{l}f\|^2_{L^2(\mu)}=\|T_n (g\pp_l f)- g T_n(\pp_l f)\|_{L^2(\mu)}^2
	 =\sum_{k\in\Z^d} \left|(1+n^{-2}|k|^2)^{-1}\widehat{g\pp_l f}(k)
	-\widehat{gT_n\pp_l f} (k)\right|^2 \\
	=&\sum_{k\in\Z^d}
	\left|\left( \frac{m_l}{1+\frac{1}{n^2} |k|^2}- \ff{m_l} {1+\frac{1}{n^2} |m|^2}\right) \sum_{m\in\Z^d} \widehat{g}(k-m)\widehat{f }(m) \right|^2\\
	=&\sum_{k\in\Z^d}
	\bigg|\sum_{j=1}^d  \sum_{m\in\Z^d} \widehat{\pp_jg}(k-m)
	\left\{\ff{\F(T_n\pp_l\pp_j f)(m) }{n^2(1+\frac{1}{n^2}|k|^2) }+ \ff{{\rm i}k_j \F(T_n \pp_lf)(m)}{n^2(1+\ff 1 n|k|^2)}\right\}
	  \bigg|^2\\
	  = & \sum_{k\in\mathbb Z^d} \left|\sum_{j=1}^d\left\{ \ff{\F\left( (\pp_jg)T_n\pp_l\pp_jf\right)(k) }{n^2(1+\frac{1}{n^2}|k|^2)} + \ff{{\rm i}k_j \F\left( (\pp_j g) T_n \pp_l f\right)(k)}{n^2(1+\frac{1}{n^2}|k|^2)}\right\}\right|^2\\
	\leq& 2d \sum_{j=1}^d \left\{\frac{1}{n^4} \big\|(\pp_j g) T_n \pp_l\pp_j f\big\|_{L^2(\mu)}^2 + \frac{1}{n^2} \big\|(\pp_j g) T_n\pp_l f\big\|_{L^2(\mu)}^2\right\}
	\le   c \|\nn g\|_\infty^2 \|f\|_{L^2(\mu)}^2,  	\end{align*}
	where the last step is due to the fact that
	\begin{align*}
	&\frac{1}{n^4}\big\|  T_n \pp_l\pp_j f\big\|_{L^2(\mu)}^2
	+ \frac{1}{n^2} \big\|  T_n\pp_l f\big\|_{L^2(\mu)}^2
	\le C\|f\|_{L^2(\mu)}^2,\ \ n\ge 1\end{align*}  holds for some constant $C>0$.
    Then we obtain \eqref{Te commutator xl} and hence finish the proof.
\end{proof}

We also need the following lemma on the commutator estimates for $D^s$.

\begin{lem}[\cite{KPV-1991-JAMS}]\label{KP commutator estimate}
	Let  $p,p_2,p_3 \in(1,\infty)$ and $p_1,p_4\in (1,\infty]$   such that
	$$\frac{1}{p}=\frac{1}{p_1}+\frac{1}{p_2}=\frac{1}{p_3}+\frac{1}{p_4}.$$ Then for any $s>0$, there  exists a constant $C>0$ such that
	$$
	\|\left[D^s,f\right]g\|_{L^p(\mu)}\leq C(\|\nabla f\|_{L^{p_1}(\mu)}\|D^{s-1}g\|_{L^{p_2}(\mu)}+\|D^sf\|_{L^{p_3}(\mu)}\|g\|_{L^{p_4}(\mu)})$$ holds for all $ f,g\in H^s\cap W^{1,\infty}(\mathbb T^d\to\R^d;\mu).$
	
\end{lem}

We are now ready to prove   Theorem \ref{T-TE}.
Let $s,s'$ be given in Assumption \ref{Assum-B}.
Take $\H=H^s$,
$\BB= H^{s'}$, $\H_0=C^\infty(\mathbb T^d;\R^d)$, and let  $J_n$ and $T_n$ be given in \eqref{JN} and \eqref{TN}, respectively. Take
\beg{equation}\label{BN}
B(t,X)=B(X)= -(X\cdot\nn)X,\ \
B_n(t,X)= B_n(X)= J_n B(J_n X),\ \ t\ge 0,\ X\in H^s.
\end{equation}
Obviously, \ref{A1} follows from \ref{B1}.
So, it remains  to verify  \ref{A2}, \ref{A3}, \ref{A4} and \ref{A5}.

\beg{proof}[Proof of  \ref{A2}]  By   \eqref{Jn property 5}, we have
$$\|B_n(t,X)\|_{H^s}\leq \|(J_{n}X\cdot\nabla)J_{n}X\|_{H^s}
\leq \|J_{n}X\|_{H^s}\|\nabla J_{n}X\|_{H^s}\leq n\,\|X\|^2_{H^s},$$
and
\begin{align*}
\|B_n(t,X)-B_n(t,Y)\|_{H^s}\le&\|(J_{n}X\cdot\nabla)J_{n}X-(J_{n}Y\cdot\nabla)J_{n}Y\|_{H^s}\\
\le&\|X\|_{H^s}\|\nabla(J_{n}X-J_{n}Y)\|_{H^s}+\|X-Y\|_{H^s}\|\nabla J_{n}Y\|_{H^s}\\
\lesssim& n\,(\|X\|_{H^s}+\|Y\|_{H^s}) \|X-Y\|_{H^s}.
\end{align*}
  Finally, by identifying $H^{s'}$ and $(H^{s'})^*$ via the Riesz isomorphism, then \ref{A2} follows from the above estimates and \eqref{Jn property 1}. \end{proof}

\begin{proof}[Proof of \ref{A3}]
It follows from Lemma \ref{KP commutator estimate}, integration by parts, $H^{s-1}\hookrightarrow W^{1,\infty}$, \eqref{Jn property 3} and \eqref{Jn property 5} that for some $C=C_s>0$,
\begin{align}
\left|\left\<B_n(X),X\right\>_{H^s}\right|
 \leq&\left|\left\<\left[D^s,
	(J_nX\cdot\nabla)J_nX\right],D^sJ_n X\right\>_{L^2(\mu)}\right|+
	\left|\left\<(J_nX\cdot\nabla)D^sJ_nX,D^sJ_n X\right\>_{L^2(\mu)}\right|\notag\\
	\leq& C_s \|J_nX\|_{H^s}\|\nabla J_nX\|_{\infty}\|J_n X\|_{H^s}
	+\|\nabla J_nX\|_{\infty}\|J_n X\|^2_{H^s}\notag\\
	\leq&  (C_s+1)\|X\|_{H^{s-1}}\|X\|^2_{H^s},\ \ X\in\H:=H^s. \label{DSP}
\end{align}
Then above estimate and \ref{B2} yields \ref{A3}.
\end{proof}

\beg{proof}[Proof of \ref{A4}]
Let  $T_n$ be  defined in  \eqref{TN}. It is easy to see that \eqref{A4-1} is satisfied. So, to verify \ref{A4} it remains to check \eqref{A4-2}.  By \eqref{Jn property 3}, \eqref{Jn property 4},  \eqref{Jn property 5},   Lemma \ref{KP commutator estimate}, integration by parts, Lemma \ref{Je commutator},   and $H^{s}\hookrightarrow W^{1,\infty}$, we find constants $c_1,c_2,c_3>0$ such that
\begin{equation*}
\begin{split}
&\left|\left\<T_n\{(X\cdot\nabla)X\},T_n X\right\>_{H^s}\right| \\
=&\left|\left\<\left[D^s,
(X\cdot\nabla)X\right],D^sT^2_n X\right\>_{L^2(\mu)}+
\left\<T_n\{(X\cdot\nabla)D^sX\},D^sT_n X\right\>_{L^2(\mu)}\right| \\
\leq&\left|\left\<\left[D^s,
(X\cdot\nabla)X\right],D^sT^2_n X\right\>_{L^2(\mu)}\right|+
\left|\left\<[T_n,(X\cdot\nabla)]D^sX,D^sT_n X\right\>_{L^2(\mu)}\right| \\
&\qquad
+\left|\left\<(X\cdot\nabla)D^sT_n X,D^sT_n X\right\>_{L^2(\mu)}\right| \\
\leq& c_1 \|X\|_{H^s}\|\nabla X\|_{\infty}\|T^2_n X\|_{H^s}
+c_2\|\nabla X\|_{\infty}\|X\|_{H^s}\|T_n X\|_{H^s} \\
\le & c_3 \|X\|^3_{H^s},\ \ X\in H^s=\H.
\end{split}\end{equation*}

Therefore,   \eqref{A4-2} holds.
\end{proof}

\beg{proof}[Proof of \ref{A5}] By \ref{B3}, for any $N\ge 1$  it suffices to find a constant $C_N>0$ such that
\beg{equation*} \left\<B(t,X)-B(t,Y),X-Y\right\>_{H^{s'}}
\leq C_N\|X-Y\|^2_{H^{s'}},\ \ X,Y\in \C_{T,H^s,N}.\end{equation*}
Let  $Z=X-Y$. By   $H^s\hookrightarrow H^{s'} \hookrightarrow W^{1,\infty}$ and Lemma \ref{KP commutator estimate}, we find  constants $c_1,c_2>0$ such that
\begin{align*}
&\left\<B(t,X)-B(t,Y),X-Y\right\>_{H^{s'}}\\
=&-\left\<(Z\cdot\nn)X,Z\right\>_{H^{s'}}
-\left\<(Y\cdot\nn)Z,Z\right\>_{H^{s'}}\\
\leq& c_1\|X\|_{H^s}\|Z\|_{H^{s'}}^2
+\left|\left\<D^{s'}\left((Y\cdot\nn)Z\right),D^{s'}Z\right\>_{L^2(\mu)}\right|\\
\le &c_1\|X\|_{H^s}\|Z\|_{H^{s'}}^2+c_2\|D^{s'}Y\|_{L^{2}(\mu)}\|\nabla Z\|_{L^\infty(\mu)}\|Z\|_{H^{s'}}
+c_2\|\nabla Y\|_{\infty}\|Z\|^2_{H^{s'}}\\
\le &c_1\|X\|_{H^s}\|Z\|_{H^{s'}}^2+c_2\|Y\|_{H^{s}}\|Z\|^2_{H^{s'}},
\end{align*}
which is the desired estimate.
 \end{proof}

\section*{Acknowledgement}
We would like to thank Mr. Wei Hong for careful check and corrections.


\begin{thebibliography}{10}
	
	\bibitem{BRW20} J. Bao, P. Ren, F.-Y. Wang, \emph{Bismut formulas for Lions derivative of McKean-Vlasov SDEs with memory,} J. Differential Equations. no. 282 (2021), 285–329.
	
	\bibitem{BFH-18-book}
	D. Breit, E. Feireisl, M. Hofmanov\'{a}, Martina,
	\emph{Stochastically forced compressible fluid flows},
	De Gruyter Series in Applied and Numerical Mathematics, (3)
	2018, xii+330.
	
	
	\bibitem{Brzezniak-etal-2005-PTRF}
	Z.~Brze\'{z}niak, B.~Maslowski, and J.~Seidler.
	\newblock Stochastic nonlinear beam equations.
	\newblock {\em Probab. Theory Related Fields}, 132(1):119--149, 2005.
	
	
	\bibitem{CH-1993-PRL}
	R.~Camassa and D.~D. Holm.
	\newblock An integrable shallow water equation with peaked solitons.
	\newblock {\em Phys. Rev. Lett.}, 71(11):1661--1664, 1993.
	
	

\bibitem{CE-1998-Acta}
A.~Constantin and J.~Escher.
\newblock Wave breaking for nonlinear nonlocal shallow water equations.
\newblock {\em Acta Math.}, 181(2):229--243, 1998.

\bibitem{CE-1998-CPAM}
A.~Constantin and J.~Escher.
\newblock Well-posedness, global existence, and blowup phenomena for a periodic
quasi-linear hyperbolic equation.
\newblock {\em Comm. Pure Appl. Math.}, 51(5):475--504, 1998.

\bibitem{CL-2009-ARMA}
A.~Constantin and D.~Lannes.
\newblock The hydrodynamical relevance of the {C}amassa-{H}olm and
{D}egasperis-{P}rocesi equations.
\newblock {\em Arch. Ration. Mech. Anal.}, 192(1):165--186, 2009.



\bibitem{DP} G. Da Prato. \emph{ Kolmogorov Equations for Stochastic PDEs,}  Adv. Courses Math. CRM Barcelona, Birkh\"auser Verlag, Basel, 2004.

	\bibitem{Prato-Zabczyk-2014-Cambridge}
G.~Da~Prato and J.~Zabczyk.
\newblock {\em Stochastic equations in infinite dimensions}, volume 152 of {\em
	Encyclopedia of Mathematics and its Applications}.
\newblock Cambridge University Press, Cambridge, second edition, 2014.

\bibitem{D-13-notes}
A. Debussche. \emph{ Ergodicity results for the stochastic Navier--Stokes equations: an introduction}, In Topics in Mathematical Fluid Mechanics, volume 2073 of Lecture Notes in Math., pages 23--108, Springer, Heidelberg, 2013.

	\bibitem{Dullin-Gottwald-Holm-2001-PRL}
H.~R. Dullin, G.~A. Gottwald, and D.~D. Holm.
\newblock An integrable shallow water equation with linear and nonlinear
dispersion.
\newblock {\em Phys. Rev. Lett.}, 87:194501, Oct 2001.

\bibitem{FF-13-JFA} E. Fedrizzi, F. Flandoli, \emph{Noise prevents singularities in linear transport equations,} J. Funct. Anal., 264, 6 (2013), 1329--1354.


\bibitem{FNO-18} E. Fedrizzi, W. Neves, C. Olivera,
\emph{On a class of stochastic transport equations for $L^2_{loc}$ vector fields,}
Ann. Sc. Norm. Super. Pisa Cl. Sci. (5) Vol. XVIII (2018), 397--419.

\bibitem{F-10-notes} F. Flandoli, \emph{Random Perturbation of PDEs and Fluid Dynamic Models,} Saint Flour Summer School Lectures 2010, Lecture Notes in Mathematics, 2015, Springer, Berlin (2011).

\bibitem{FGP-10-Inven} F. Flandoli, M. Gubinelli, E. Priola, \emph{Well-posedness of the transport equation by stochastic perturbation,} Invent. Math. 180 (1) (2010) 1--53.

\bibitem{FGP-11-SPA} F. Flandoli, M. Gubinelli, E. Priola,
\emph{Full well-posedness of point vortex dynamics corresponding to stochastic 2D Euler equations,}
Stochastic Process. Appl. 121 (2011), 1445--1463.


\bibitem{FF-1981-PhyD}
B.~Fuchssteiner and A.~S. Fokas.
\newblock Symplectic structures, their {B}\"{a}cklund transformations and
hereditary symmetries.
\newblock {\em Phys. D}, 4(1):47--66, 1981/82.

\bibitem{HS-2004-PhyA}
D.~D. Holm and M.~F. Staley.
\newblock Nonlinear balance and exchange of stability of dynamics of solitons,
peakons, ramps/cliffs and leftons in a {$1+1$} nonlinear evolutionary {PDE}.
\newblock {\em Phys. Lett. A}, 308(5-6):437--444, 2003.



	\bibitem{Gawarecki-Mandrekar-2010-Springer}
L.~Gawarecki and V.~Mandrekar.
\newblock {\em Stochastic differential equations in infinite dimensions with
	applications to stochastic partial differential equations}.
\newblock Probability and its Applications (New York). Springer, Heidelberg,
2011.


\bibitem{GV-14-AoP} N. Glatt-Holtz, V. Vicol, \emph{Local and global existence of smooth solutions for the stochastic Euler equations with multiplicative noise, } Ann. Probab. 42 (1) (2014), 80--145.

\bibitem{GLS-2019-JMFM}
G.~Gui, Y.~Liu, and J.~Sun.
\newblock A nonlocal shallow-water model arising from the full water waves with
the {C}oriolis effect.
\newblock {\em J. Math. Fluid Mech.}, 21(2):Paper No. 27, 29, 2019.

\bibitem{HK-09-DIE}  A. Himonas, C. Kenig, \emph{Non-uniform dependence on initial data for the CH equation on the line,}
Diff. Integr. Eqns. 22 (2009), 201--224.

	
	\bibitem{HRW19} X. Huang, M.  R\"ockner, F.-Y.  Wang,   \emph{Nonlinear
		Fokker-Planck equations for probability measures on path space and
		path-distribution dependent SDEs,}    Discrete Contin. Dyn. Syst.   39(2019), 3017--3035.
	
\bibitem{HW19}  	X. Huang, F.-Y. Wang, \emph{Distribution dependent SDEs with singular coefficients, } Stoch. Proc. Appl. 129(2019), 4747--4770.		







\bibitem{KPV-1991-JAMS}
C.~E. Kenig, G.~Ponce, and L.~Vega.
\emph{Well-posedness of the initial value problem for the {K}orteweg-de
{V}ries equation,}
J. Amer. Math. Soc., 4(2):323--347, 1991.

\bibitem{Hasminskii-1969-Book}
R.~Z. Khas'minskii.
\newblock {\em Stability of systems of differential equations under random
	perturbations of their parameters. {\rm (Russian)}}.
\newblock Izdat. ``Nauka'', Moscow, 1969.

\bibitem{K-10-JFA}  J. U. Kim, \emph{On the Cauchy problem for the transport equation with random noise,}  J. Funct. Anal., 259 (2010), 3328--3359.

\bibitem{Krylov-Rozovskiui-1979-chapter}
N.~V. Krylov and B.~L. Rozovski\u{\i}.
\newblock Stochastic evolution equations.
\newblock In {\em Current problems in mathematics, {V}ol. 14 ({R}ussian)},
pages 71--147, 256. Akad. Nauk SSSR, Vsesoyuz. Inst. Nauchn. i Tekhn.
Informatsii, Moscow, 1979.




\bibitem{KS-12-book}
S. Kuksin, A. Shirikyan,
\emph{Mathematics of two-dimensional turbulence},
 Cambridge University Press, Cambridge, 2012, xvi+320.

\bibitem{K-14-ECP} T. Kurtz, \emph{Weak and strong solutions of general stochastic models,} Electron. Commun. Probab. 19 (2014), paper no. 58, 16 pp.



\bibitem{MO-17-AMPA} D. Mollinedo, C. Olivera,
\emph{Stochastic continuity equation with nonsmooth velocity,}
Ann. Mat. Pura Appl. (4) 196 (2017), 1669--1684.

\bibitem{MO-17-BBMS} D. Mollinedo, C. Olivera,
\emph{Well-posedness of the stochastic transport equation with unbounded drift,}
Bull. Braz. Math. Soc. (N.S.)  48 (2017), 663--677.

\bibitem{NO-15-NODEA} W. Neves, C. Olivera,
\emph{Wellposedness for stochastic continuity equations with Ladyzhenskaya-Prodi-Serrin condition,}
NoDEA Nonlinear Differential Equations Appl. 22 (2015), 1247--1258.

	\bibitem{Pard} E. Pardoux, \emph{Sur des equations aux d\'eriv\'ees
		partielles stochastiques monotones,} C. R. Acad. Sci. 275(1972),
	A101--A103.
	
	
	
	
	\bibitem{PR-2007-book}  C. Pr\'{e}v\^{o}t,  M.  R\"{o}ckner,
\emph{A Concise Course on Stochastic Partial Differential Equations},
 Lecture Notes in Mathematics,
 1905,
 Springer, Berlin,
2007.
	
	
	
	\bibitem{RW19} P. Ren, F.-Y. Wang,  \emph{Bismut formula for Lions derivative of distribution dependent SDEs and applications,}   J.
	Differential Euqations,   267(2019), 4745--4777.
	
	\bibitem{RW20} P.  Ren, F.-Y.  Wang,  \emph{Donsker-Varadhan large deviations for path-distribution dependent SPDEs, }  J. Math. Anal. Appl. no. 499 (2021), Doi:10.1016/j.jmaa.2021.125000.
	

\bibitem{RZZ-14-SPA}
M. R\"{o}ckner, R. Zhu, X. Zhu, \emph{Local existence and non-explosion of solutions for stochastic fractional partial differential
equations driven by multiplicative noise,}
Stoch. Proc. Appl., 124 (2014), 1974--2002.
	
	\bibitem{T-18-SIMA} H. Tang, \emph{On the pathwise solutions to the Camassa--Holm equation with multiplicative noise,} SIAM. J. Math. Anal.,  50(1):1322--1366, 2018.


	
	\bibitem{W18} F.-Y. Wang, \emph{Distribution dependent SDEs for Landau type equations,} Stoch. Proc. Appl. 128(2018), 595--621.
	
	
\bibitem{ZLM-20-AMPA}
	M.  Zhu, Y. Liu,  Y. Mi,
	\emph{Wave-breaking phenomena and persistence properties for the
			nonlocal rotation-{C}amassa--{H}olm equation},
 Ann. Mat. Pura Appl. (199), 2020, 355--377.
	


\end{thebibliography}
\end{document}